\pdfoutput=1
\documentclass[11pt]{amsart}

\usepackage{graphicx}
\usepackage[english]{babel}
\usepackage[T1]{fontenc}
\usepackage[latin1]{inputenc}
\usepackage{amsfonts}
\usepackage{amssymb}
\usepackage{amsthm}
\usepackage{amsmath}
\usepackage{tikz}
\usetikzlibrary{decorations.markings,arrows,calc,math}
\usepackage{float}
\usepackage{enumerate}
\usepackage{accents}
\usepackage{mathtools}
\usepackage{mathrsfs}
\usepackage{comment}
\usepackage{afterpage}
\usepackage{bbm}
\usepackage{dsfont}
\usepackage{stmaryrd}
\usepackage{hyperref}
\usepackage{mleftright}
\usepackage{subfig}
\usepackage{soul}
\usepackage{tikz-cd} 
\usepackage{fullpage}

\theoremstyle{plain}
\newtheorem{thm}{Theorem}[section]
\newtheorem{lem}[thm]{Lemma}
\newtheorem{prop}[thm]{Proposition}
\newtheorem{cor}[thm]{Corollary}
\newtheorem*{thm*}{Theorem}
\newtheorem*{prop*}{Proposition}
\newtheorem*{cor*}{Corollary}
\newtheorem{thmintro}{Theorem}

\newtheorem{corintro}[thmintro]{Corollary}
\newtheorem{propintro}[thmintro]{Proposition}

\theoremstyle{definition}
\newtheorem{defn}[thm]{Definition}
\newtheorem{ex}[thm]{Example}
\newtheorem{rmk}[thm]{Remark}
\newtheorem*{rmk*}{Remark}

\newtheorem*{quest*}{Question}

\newtheorem*{defn*}{Definition}

\newcommand{\acts}{\curvearrowright}
\newcommand{\ra}{\rightarrow}
\newcommand{\Ra}{\Rightarrow}

\newcommand{\wt}{\widetilde}
\newcommand{\wh}{\widehat}
\newcommand{\x}{\times}
\renewcommand{\o}{\circ}
\newcommand{\id}{\mathrm{id}}
\newcommand{\mbb}{\mathbb}
\newcommand{\mc}{\mathcal}
\newcommand{\mf}{\mathfrak}
\newcommand{\mscr}{\mathscr}

\newcommand{\Z}{\mathbb{Z}}

\newcommand{\s}{\sigma}
\newcommand{\eps}{\epsilon}

\renewcommand{\L}{\Lambda}
\newcommand{\g}{\gamma}
\newcommand{\G}{\Gamma}

\newcommand{\E}{\mbb{E}}
\newcommand{\CAT}{{\rm CAT(0)}}
\DeclareMathOperator{\Fix}{Fix}

\DeclareMathOperator{\lk}{lk}
\DeclareMathOperator{\St}{st}
\DeclareMathOperator{\Isom}{Isom}

\DeclareMathOperator{\Aut}{Aut}
\DeclareMathOperator{\Out}{Out}

\newcommand{\sq}{\subseteq}

\renewcommand{\S}{\mathbb{S}}
\newcommand{\ora}{\overrightarrow}
\newcommand{\ola}{\overleftarrow}

\newcommand{\blue}[1]{\textcolor{blue}{#1}}

\newcommand{\red}[1]{\textcolor{red}{#1}}

\begin{document}

\title{Divisible cube complexes \\ and finite-order automorphisms of RAAGs} 
\author[E.\,Fioravanti]{Elia Fioravanti}\address{Institute of Algebra and Geometry, Karlsruhe Institute of Technology}\email{elia.fioravanti@kit.edu} 
\thanks{The author is supported by Emmy Noether grant 515507199 of the Deutsche Forschungsgemeinschaft (DFG)}

\begin{abstract}
We give a geometric characterisation of those groups that arise as fixed subgroups of finite-order untwisted automorphisms of right-angled Artin groups (RAAGs). They are precisely the fundamental groups of a class of compact special cube complexes that we term ``divisible''. 

The main corollary is that surface groups arise as fixed subgroups of finite-order automorphisms of RAAGs, as do all commutator subgroups of right-angled Coxeter groups. These appear to be the first examples of such fixed subgroups that are not themselves isomorphic to RAAGs.

Using a variation of canonical completions, we also observe that every special group arises as the fixed subgroup of an automorphism of a finite-index subgroup of a RAAG.
\end{abstract}

\maketitle


\section{Introduction}

Given an automorphism $\varphi$ of a group $G$, it is natural to wonder about the properties of its subgroup of fixed points $\Fix(\varphi)\leq G$. A complete picture of the structure of such fixed subgroups is available only when $G$ satisfies strong constraints, and it is not unusual for fixed subgroups to behave much more wildly than the ambient group $G$.

Fixed subgroups have been extensively studied when $G$ is either a free group or the fundamental group of a closed surface. In these cases, it can be shown that $\Fix(\varphi)$ is always a free group (provided that $\varphi\neq\id$) and that its ``complexity'' never exceeds that of $G$; as a measure of complexity, one can take the modulus of the Euler characteristic. When $G$ is a closed surface group, this fact can be quickly deduced from the existence of Nielsen--Thurston decompositions for surface homeomorphisms \cite{Jaco-Shalen}. When $G$ is free, this proved to be a more complex problem known as the Scott Conjecture, which was the focus of a large body of work in the 80s \cite{Ger83,Goldstein-Turner,Ger87,Cohen-Lustig} before being finally solved by Bestvina and Handel in \cite{Bestvina-Handel}. 

The special case when an automorphism $\varphi\in\Aut(F_n)$ has finite-order projection $[\varphi]\in\Out(F_n)$ is considerably simpler to treat. In this situation, $\Fix(\varphi)$ is always a free factor of $F_n$ \cite{Dyer-Scott,Culler}, a property that is rarely satisfied for general automorphisms of free groups.

\medskip
Part of the reason why fixed subgroups are so restricted when $G$ is a free or surface group is that such low-dimensional groups do not have very interesting subgroups: every infinite-index subgroup of $G$ is free. This leads to a natural desire to study fixed subgroups in greater generality.

Right-angled Artin groups (RAAGs) are an ideal candidate for this. On the one hand, their subgroups are known to be extremely rich, including e.g.\ hyperbolic $3$--manifold groups \cite{Bergeron-Wise,Kahn-Markovic}, many non-geometric $3$--manifold groups \cite{Przytycki-Wise}, random groups \cite{Ollivier-Wise} and finitely presented small cancellation groups \cite{Wise04}. 

On the other hand, automorphisms of RAAGs are also plentiful \cite{Laurence,Servatius} and they have been the object of a burgeoning theory built along the blueprint of the classical theory of $\Out(F_n)$. For instance, analogues of Outer Space are available \cite{CCV,CSV,BCV}, as well as a peak reduction algorithm \cite{Day} and a good understanding of virtual cohomological dimension \cite{CV09,Day-Wade,Day-Wade-Sale}.

General automorphisms of RAAGs can have rather wild fixed subgroups: all Bestvina--Brady groups \cite{Bestvina-Brady} can occur\footnote{More precisely, each RAAG of the form $A_{\G}\x\Z$ has an automorphism with fixed subgroup $BB_{\G}\x\Z$; see \cite[Example~4.13]{Fio10a} for details.} and many of them are not even finitely generated. At the same time, the large class of automorphisms of RAAGs known as \emph{untwisted automorphisms} does have well-behaved fixed subgroups \cite[Theorem~C]{Fio10a}, which turn out to always be special groups in the Haglund--Wise sense \cite{HW08}.

Our main result (Theorem~\ref{thmintro:divisible_vs_fix}) is that --- in sharp contrast to the situation for automorphisms of free groups --- there are not many more restrictions on the isomorphism types of fixed subgroups of untwisted automorphisms of RAAGs, not even when the latter are assumed to have finite order. Indeed, a group arises as a fixed subgroup of such an automorphism exactly when it is the fundamental group of a ``divisible'' compact special cube complex.

As a first approximation, the reader can think of ``divisible'' cube complexes as those special cube complexes $Q$ that satisfy the following additional requirements:
\begin{enumerate}
\item any two distinct vertices of $Q$ can be separated by a union of pairwise-disjoint hyperplanes;
\item any hyperplane half-carrier in $Q$ can be separated from each vertex outside it by a union of pairwise-disjoint hyperplanes of $Q$.
\end{enumerate}
What we have just defined is actually the class of ``strongly divisible'' cube complexes (Definition~\ref{defn:strongly_divisible}), which would suffice for the first two examples in Corollary~\ref{corintro:examples} below, but importantly not for the equivalence in Theorem~\ref{thmintro:divisible_vs_fix}. The actual definition of a ``divisible'' cube complex is a little weaker and more technical, so we delay it until Section~\ref{sect:meat} (see Definition~\ref{defn:divisible}).

\begin{thmintro}\label{thmintro:divisible_vs_fix}
The following properties are equivalent for a group $H$:
\begin{enumerate}
\setlength\itemsep{.25em}
\item $H$ is the fundamental group of a divisible, compact, special cube complex;
\item there exist a right-angled Artin group $A_{\G}$ and a finite-order, pure, untwisted outer automorphism $[\varphi]\in\Out(A_{\G})$ such that $H\cong\Fix(\varphi)$, for some representative $\varphi\in\Aut(A_{\G})$.
\end{enumerate}
\end{thmintro}

The two arrows of the above equivalence are proved in Theorem~\ref{thm:divisible->fix} and Proposition~\ref{prop:fix->divisible}, respectively. Our proof was inspired by the construction of Salvetti blowups from \cite{CSV} and it also relies in an essential way on the fixed point theorem for (untwisted) Outer Space from \cite{BCV2}. 

Roughly, the idea for the implication $(1)\Ra(2)$ is that a divisible cube complex $Q$ can always be embedded into a ``host'' cube complex that strongly resembles one of the Salvetti blowups from \cite{CSV} (despite not quite being one). The fundamental group of the host is a RAAG and the host admits a natural cubical automorphism having the image of $Q$ as a connected component of its fixed set. Conversely, the implication $(2)\Ra(1)$ is proved by realising finite-order outer automorphisms as isometries of a Salvetti blowup and exploiting the particular geometry of the latter.

Despite the slightly convoluted definition, divisible cube complexes are plentiful in nature, which leads to the following remarkable consequence (see Examples~\ref{ex:strongly_divisible} and~\ref{ex:braid_divisible}).

\begin{corintro}\label{corintro:examples}
The following groups arise as fixed subgroups of finite-order, pure, untwisted automorphisms of right-angled Artin groups:
\begin{enumerate}
\item[(a)] the fundamental group of any closed orientable surface;
\item[(b)] the commutator subgroup of any right-angled Coxeter group;
\item[(c)] any graph braid group.
\end{enumerate}
\end{corintro}

To the best of my knowledge, all previously known fixed subgroups of finite-order (or untwisted) automorphisms of RAAGs happened to be themselves isomorphic to RAAGs. 

We mention that among commutator subgroups of right-angled Coxeter groups one finds fundamental groups of finite-volume hyperbolic manifolds of all dimensions $\leq 8$ \cite{Loebell,Everitt,PV,ERT}, 
as well as Gromov-hyperbolic groups of arbitrarily high dimension \cite{JS03,Osajda}.

Graph braid groups were first considered in Abrams' thesis \cite{Abrams} and it was later observed by Crisp and Wiest that they are fundamental groups of compact special cube complexes \cite{CW04}. The definition of graph braid groups is simple: for any connected, finite, simplicial graph $\G$ and any integer $0\leq n\leq \#\G^{(0)}$, one considers the fundamental group $B_n(\G)$ of the configuration space of $n$ distinct and unmarked points in $\G$. We refer the reader to Genevois' work \cite{Gen-IJAC} for an excellent introduction to the subject. 

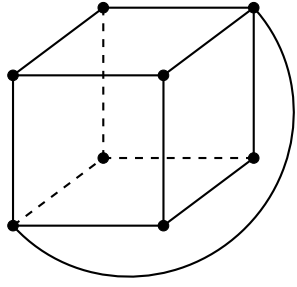
\begin{figure}
\begin{tikzpicture}[baseline=(current bounding box.center)]
\draw[fill] (0,0) circle [radius=0.07cm];
\draw[fill] (0,2) circle [radius=0.07cm];
\draw[fill] (2,0) circle [radius=0.07cm];
\draw[fill] (2,2) circle [radius=0.07cm];
\draw[fill] (1.2,2.9) circle [radius=0.07cm];
\draw[fill] (3.2,2.9) circle [radius=0.07cm];
\draw[fill] (1.2,0.9) circle [radius=0.07cm];
\draw[fill] (3.2,0.9) circle [radius=0.07cm];
\draw[thick] (0,0) -- (0,2);
\draw[thick] (0,2) -- (2,2);
\draw[thick] (2,2) -- (2,0);
\draw[thick] (2,0) -- (0,0);
\draw[thick] (0,2) -- (1.2,2.9);
\draw[thick] (1.2,2.9) -- (3.2,2.9);
\draw[thick] (3.2,2.9) -- (2,2);
\draw[thick] (3.2,2.9) -- (3.2,0.9);
\draw[thick] (3.2,0.9) -- (2,0);
\draw[thick,dashed] (0,0) -- (1.2,0.9);
\draw[thick,dashed] (1.2,0.9) -- (1.2,2.9);
\draw[thick,dashed] (1.2,0.9) -- (3.2,0.9);
\tikzset{
    partial ellipse/.style args={#1:#2:#3}{
        insert path={+ (#1:#3) arc (#1:#2:#3)}
    }
}
\draw[thick,rotate around={43:(1.6,1.45)}] (1.6,1.45) [partial ellipse=180:360:2.16cm and 2.1cm];
\end{tikzpicture}
\caption{A graph $\L$ such that the RAAG $A_{\L}$ admits an involution $\varphi$ with $\Fix(\varphi)$ isomorphic to the genus--$2$ surface group.}
\label{fig:intro}
\end{figure}

Explicit automorphisms having the groups in Corollary~\ref{corintro:examples} as fixed subgroups are given in Example~\ref{ex:explicit}. In fact, each of these groups is the fixed subgroup of an involution of a right-angled Artin group $A_{\L}$, where the graph $\L$ admits the following pleasant characterisations.
\begin{enumerate}
\setlength\itemsep{.25em}
\item[(a)] For the fundamental group of the (say) genus--$2$ surface: the graph $\L$ is pictured in Figure~\ref{fig:intro}. 
\item[(b)] For the commutator subgroup of the right-angled Coxeter group $W_{\G}$: the graph $\L$ has two vertices $v_1,v_2$ for every vertex $v\in\G$; the vertices $\{v_1\mid v\in\G\}$ span a copy of $\G$ within $\L$, whereas the vertices $\{v_2\mid v\in\G\}$ span a whole clique, and $\L$ has additional edges $[v_1,w_2]$ whenever $v,w\in\G$ are distinct.
\item[(c)] For the graph braid group $B_n(\G)$: the graph $\L$ has a vertex $\overline e$ for every edge $e\sq\G$ and a vertex $\overline v$ for every vertex $v\in\G$; there are edges $[\overline e_1,\overline e_2]$ when $e_1\cap e_2=\emptyset$, edges $[\overline v_1,\overline v_2]$ whenever $v_1,v_2$ are distinct, and finally edges $[\overline e,\overline v]$ exactly when $v\not\in e$.

Note that $\L$ is independent of the value of $n$ and, in fact, there are involutions $\varphi_n$ of $A_{\L}$ for all $0\leq n\leq\#\G^{(0)}$ so that $\Fix(\varphi_n)\cong B_n(\G)$. Moreover, the involutions $\varphi_n\in\Aut(A_{\L})$ all descend to the same element of $\Out(A_{\L})$, which is independent\footnote{We emphasise that this does \emph{not} yield any isomorphisms between $B_n(\G)$ and $B_m(\G)$ for $m\neq n$; see Lemma~\ref{lem:pi1_vs_fix} for a general discussion on fixed subgroups of automorphisms in the same outer class.} of the value of $n$. 
\end{enumerate}

\smallskip
I do not know how strong of a requirement divisibility of a cube complex actually is. Not every special cube complex is divisible and divisibility is not always preserved when passing to finite covers (see Example~\ref{ex:non-examples}). At the same time, things might be different when it comes to fundamental groups and it is theoretically possible that every special group is the fundamental group of a divisible cube complex, though there is not much evidence for this at the moment. 

I particularly want to draw attention to the following questions, for which I do not have answers.

\begin{quest*}
\begin{enumerate}
\item[]
\item Is there a group $H$ that is the fundamental group of a compact special cube complex, but not the fundamental group of a divisible one? 
\item Does every compact special cube complex admit a divisible finite cover? 
\end{enumerate}
\end{quest*}

Finally, it is interesting to note that, dropping the divisibility requirement, every special group $H$ can be realised as the fixed subgroup of an automorphism of a \emph{finite-index} subgroup of a RAAG.

In addition, such automorphisms can be taken to coarsely preserve the coarse median structure of the ambient group. For automorphisms of RAAGs, this property is equivalent to the automorphism being untwisted \cite{Fio10a}. Thus, being ``coarse-median preserving'' can be viewed as an extension of untwistedness to automorphisms of finite-index subgroups of RAAGs (and, more generally, to automorphisms of general cocompactly cubulated groups).

\begin{propintro}\label{propintro:virtual_aut}
Let $H$ be the fundamental group of a compact special cube complex. There exist a finite-index subgroup $R$ of a right-angled Artin group and a finite-order, coarse-median preserving automorphism $\varphi\in\Aut(R)$ such that $\Fix(\varphi)\cong H$.
\end{propintro}

The proof of Proposition~\ref{propintro:virtual_aut} is based on a rather small variation of the classical construction, due to Haglund and Wise, of the canonical completion of a special cube complex \cite[Section~6]{HW08}. See Section~\ref{subsec:virt_RAAG} for details.

I emphasise that the converse to Proposition~\ref{propintro:virtual_aut} remains unknown. More generally, I do not know if fixed subgroups of coarse-median preserving automorphisms of special groups $G$ are themselves (virtually) special. This is true when $G$ is a right-angled Artin/Coxeter group by \cite[Theorem~C]{Fio10a}, but the argument there does not apply to automorphisms of finite-index subgroups. At the same time, fixed subgroups of coarse-median preserving automorphisms are always median-cocompact (see \cite[Lemma~2.35]{Fio10a} and \cite[Section~2.3]{FLS}), and it appears likely that median-cocompact subgroups of special groups will themselves be virtually special. The latter requires care, however, as passing to a median subalgebra causes inter-osculations in general.

\smallskip
{\bf Acknowledgements.}  I am grateful to Anthony Genevois for comments on Proposition~\ref{propintro:virtual_aut}, to Sam Shepherd for suggesting part of Example~\ref{ex:non-examples}, and to the anonymous referee for their many helpful suggestions.

\section{Preliminaries}

\subsection{Cube complexes}\label{subsec:cube_complexes}

The purpose of this subsection is to fix terminology and notation. We refer the reader to \cite{HW08} for basics on cube complexes sufficient for the rest of the paper.

Let $Q$ be a cube complex. We denote by $\Aut(Q)$ the group of \emph{cubical automorphisms} of $Q$, i.e.\ homeomorphisms of $Q$ preserving the cellular structure. When $Q$ is compact, $\Aut(Q)$ is finite.

We say that two hyperplanes $H,K\sq Q$ \emph{cross} (or equivalently, that they are \emph{transverse}) if we have $H\cap K\neq\emptyset$ and $H\neq K$. A hyperplane $H$ \emph{self-intersects} if it contains both mid-cubes of some square of $Q$. A hyperplane that does not self-intersect is \emph{embedded}.

Every hyperplane $H\sq Q$ inherits from $Q$ a natural structure of a cube complex, its cells being its intersections with the cells of $Q$. We will therefore speak of vertices of $H$ (i.e.\ intersections between $H$ and an edge of $Q$), edges of $H$ (intersections between $H$ and a square of $Q$) and so on. If $H_1,\dots,H_k\sq Q$ are hyperplanes and $C$ is a connected component of $H_1\cap\dots\cap H_k$, we similarly obtain a structure of a cube complex on $C$. If $Q$ is non-positively curved, so are its hyperplanes and intersections of hyperplanes.

The \emph{carrier} of a hyperplane $H\sq Q$ is the union $C(H)\sq Q$ of all closed cubes of $Q$ that intersect $H$. The \emph{interior} $C^{\o}(H)$ of the carrier is the union of all open cubes of $Q$ that intersect $H$. We say that an embedded hyperplane $H$ is \emph{2--sided} if $H$ disconnects $C^{\o}(H)$, necessarily into exactly two components. We say that $H$ is a \emph{carrier retract} if $C(H)$ is isomorphic to the product $H\x[0,1]$.

An \emph{oriented hyperplane} $\ora{H}$ is the data of a $2$--sided hyperplane $H\sq Q$ together with a choice of connected component of $C^{\o}(H)-H$. There are exactly two possible orientations on each $2$--sided hyperplane $H\sq Q$ and we will usually denote them by $\ora{H}$ and $\ola{H}$. We write $\mscr{W}(Q)$ for the set of all hyperplanes of $Q$ and $\mscr{O}(Q)$ for the set of oriented hyperplanes.

The \emph{positive carrier} of an oriented hyperplane $\ora{H}$ is the subcomplex $C(\ora{H})\sq C(H)$ formed by all closed cubes of $C(H)-C^{\o}(H)$ that are contained in the closure of the chosen component of $C^{\o}(H)-H$. We also call $C(\ola{H})$ the \emph{negative carrier} of the oriented hyperplane $\ora{H}$, and generally refer to the positive and negative carrier of $\ora{H}$ as the two \emph{half-carriers} of $H$. We always have 
\[C(H)=C^{\o}(H)\sqcup(C(\ora{H})\cup C(\ola{H})),\] 
where the union $C(\ora{H})\cup C(\ola{H})$ contains every vertex of $C(H)$. A hyperplane $H$ is a carrier retract if and only if it is $2$--sided and $C(\ora{H})$ and $C(\ola{H})$ are disjoint. In general, is possible for $C(\ora{H})$ and $C(\ola{H})$ to coincide, for instance if every edge of $Q$ dual to $H$ is a loop.

When $H$ is $2$--sided, every edge of $Q$ dual to $H$ has an endpoint in $C(\ora{H})$ and an endpoint in $C(\ola{H})$. This gives two surjective cubical maps $H\twoheadrightarrow C(\ora{H})$ and $H\twoheadrightarrow C(\ola{H})$.
In particular, we can define a new cube complex $Q^H$ by \emph{collapsing} the hyperplane $H$: this amounts to considering the cube complex $Q-C^{\o}(H)$ and gluing its two subcomplexes $C(\ora{H})$ and $C(\ola{H})$ to each other, identifying pairs of points that are images of the same point of $H$ under the above two surjections.

There is a natural quotient map $Q\twoheadrightarrow Q^H$ that squashes every cube of $Q$ intersecting $H$ onto its two (identified) codimension--$1$ faces in $C(\ora{H})$ and $C(\ola{H})$. When $H$ is a carrier retract, this map $Q\twoheadrightarrow Q^H$ is a homotopy equivalence

\subsection{Special cube complexes}\label{subsec:special}

Let $Q$ be a cube complex. A hyperplane $H\sq Q$ \emph{self-osculates} if it is dual to two edges $e,f$ that share a vertex $x$. When $H$ is $2$--sided, we can choose an orientation $\ora{H}$, and orient $e$ and $f$ away from $\ola{H}$ and towards $\ora{H}$. We can then distinguish between \emph{direct self-osculations} (if both or neither of $e,f$ point towards their shared vertex $x$) and \emph{indirect self-osculations} (if exactly one of $e,f$ points towards $x$). Two hyperplanes $H,K\sq Q$ \emph{inter-osculate} if they simultaneously cross and are dual to edges $e,f$ that intersect but do not span a square.

As has now become standard, we adopt the definition of specialness appearing in \cite{Wise-book} (in the original paper \cite{HW08} such cube complexes are instead called \emph{``$A$--special''}).

\begin{defn}\label{defn:special}
We say that a cube complex $Q$ is \emph{special} if it admits a locally isometric immersion into a (right-angled) Salvetti complex. Equivalently, its hyperplanes are embedded and $2$--sided, do not directly self-osculate and do not inter-osculate. We allow indirect self-osculations.
\end{defn}

If $Q$ is a special cube complex, we can consider the \emph{crossing graph} $\G_Q$ of its hyperplanes. Vertices of $\G_Q$ are hyperplanes of $Q$, with two hyperplanes connected by an edge if they cross. The standard Haglund--Wise construction shows that $Q$ admits a locally isometric immersion into the Salvetti complex $\mbb{S}_{\G_Q}$ of the RAAG $A_{\G_Q}$ (see \cite[Lemma~4.1]{HW08}).

However, this is not the most general form of a locally isometric immersion of $Q$ into a Salvetti complex, and sometimes it is possible to immerse $Q$ into smaller complexes. For this reason, we introduce the following concept (analogous e.g.\ to \cite[Definition~3.1]{Gen-relhyp}), which will play a key role in the definition of divisible cube complexes later in the paper (Definition~\ref{defn:divisible}). 

Recall that $\mscr{W}(Q)$ and $\mscr{O}(Q)$ are, respectively, the set of hyperplanes and oriented hyperplanes of the cube complex $Q$. If $\G$ is a graph, we denote by $\G^{\pm}$ the set consisting of two copies of every vertex $v\in\G^{(0)}$, labelled $v^-$ and $v^+$ respectively.

\begin{defn}\label{defn:special_colouring}
Let $Q$ be a special cube complex. A \emph{special colouring} of $Q$ is the data of a simplicial graph $\G$ and two maps\footnote{We denote both maps by the same letter and implicitly assume the first to only take vertices of $\G$ as values.} $\kappa\colon\mscr{W}(Q)\ra\G$ and $\kappa\colon\mscr{O}(Q)\ra\G^{\pm}$ satisfying the following conditions:
\begin{enumerate}
\setlength\itemsep{.25em}
\item if $\kappa(H)=v$ for a hyperplane $H\in\mscr{W}(Q)$, then $\{\kappa(\ora{H}),\kappa(\ola{H})\}=\{v^-,v^+\}$;
\item if $\ora{H},\ora{K}\in\mscr{O}(Q)$ are distinct oriented hyperplanes with $\kappa(\ora{H})=\kappa(\ora{K})$, then ${C(\ora{H})\cap C(\ora{K})=\emptyset}$; 
\item if $H,K\in\mscr{W}(Q)$ are hyperplanes with $C(H)\cap C(K)\neq\emptyset$, then $H$ and $K$ cross in $Q$ if and only if $\kappa(H)$ and $\kappa(K)$ are adjacent in $\G$.
\end{enumerate}
We denote special colourings simply as pairs $(\kappa,\G)$ and we refer to vertices of $\G$ as \emph{colours}. We say that a colouring is \emph{minimal} if it satisfies the following additional condition:
\begin{enumerate}
\setcounter{enumi}{3}
\item every vertex of $\G$ lies in the image of $\kappa$ and, for every edge $[v,w]\sq\G$, there exists at least one hyperplane in $\kappa^{-1}(v)$ that crosses a hyperplane in $\kappa^{-1}(w)$.
\end{enumerate}
\end{defn}

In other words, condition~(2) is saying that each subset $\kappa^{-1}(v)\sq\mscr{W}(Q)$ does not self-osculate ``as a family'', and condition~(3) asks that no two $\kappa^{-1}(v)$ and $\kappa^{-1}(w)$ inter-osculate\footnote{Here we are thinking of the following extension of the usual concept of inter-osculating pairs of hyperplanes: two sets of hyperplanes $\mc{U},\mc{V}\sq\mscr{W}(Q)$ ``inter-osculate'' if there exist two pairs $(\mf{u}_1,\mf{v}_1)$ and $(\mf{u}_2,\mf{v}_2)$ in $\mc{U}\x\mc{V}$ such that $\mf{u}_1$ and $\mf{v}_1$ cross, while $\mf{u}_2$ and $\mf{v}_2$ osculate (i.e.\ they are dual to edges that share a vertex without spanning a square).}. Condition~(4) can be added to ensure that $\G$ does not have any ``useless'' vertices or edges. Clearly, every special colouring can be promoted to a minimal one by removing some vertices and edges of $\G$.

Compared to $\kappa\colon\mscr{W}(Q)\ra\G$, the map $\kappa\colon\mscr{O}(Q)\ra\G^{\pm}$ can simply be viewed as a choice of orientation for every hyperplane (namely, the oriented hyperplane that gets mapped to an element of the form $v^+$). In practice, we will also think of this as a labelling of (unoriented) edge germs at each vertex of $Q$: if $e$ is an oriented edge dual to an oriented hyperplane $\ora{H}$ with $\kappa(\ora{H})=v^+$, we label the initial germ of $e$ by $v^+$ and its terminal germ by $v^-$.

\begin{defn}
A minimal special colouring is \emph{standard} if the map $\kappa\colon\mscr{W}(Q)\ra\G^{(0)}$ is a bijection.
\end{defn}

For a standard special colouring, the graph $\G$ is always isomorphic to the crossing graph $\G_Q$ of the hyperplanes of $Q$. However, in general, there are also other special colourings, where $\G$ is a proper quotient of $\G_Q$ and $\kappa$ is not injective. This simple observation is what underlies the difference between ``divisible'' and ``strongly divisible'' cube complexes in Section~\ref{subsec:divisible}.

The proof of the following lemma is classical and it is left to the reader. This result is not needed in the rest of the paper and we only include it here for motivation (also see \cite[Theorem~4.1]{Gen-relhyp}).

\begin{lem}
Let $Q$ be a special cube complex.
\begin{enumerate}
\setlength\itemsep{.25em}
\item If there exists a special colouring $\kappa\colon\mscr{W}(Q)\ra\G$, then $Q$ admits a locally isometric immersion into the Salvetti complex $\mbb{S}_{\G}$.
\item If $\rho\colon Q\looparrowright Q'$ is a locally isometric immersion with $Q'$ special, then $Q$ inherits a special colouring $\kappa\colon\mscr{W}(Q)\ra\G_{Q'}$, where each preimage $\kappa^{-1}(H')$ is the set of connected components of the preimage $\rho^{-1}(H')$ of the hyperplane $H'\sq Q'$ (each of which is a hyperplane of $Q$).
\end{enumerate}
\end{lem}

\subsection{Fixed subgroups}

The following is a classical observation (see e.g.\ Theorem~3.1 and Theorem~4.1 in \cite{Culler} for analogous statements about graphs).

\begin{lem}\label{lem:pi1_vs_fix}
Consider a non-positively curved metric space $X$ and a subgroup $\mf{F}\leq\Isom(X)$. 
\begin{enumerate}
\setlength\itemsep{.25em}
\item Suppose that $X$ is complete and $\mf{F}$ is finite, realising a subgroup $F\leq\Out(\pi_1(X))$. Then there  is a natural $1$--to--$1$ correspondence between connected components of $\Fix(\mf{F})\sq X$ and lifts of $F$ to finite subgroups $\wt F\leq\Aut(\pi_1(X))$ up to conjugacy by inner automorphisms of $\pi_1(X)$. In particular, such a lift $\wt F$ exists if and only if $\Fix(\mf{F})\neq\emptyset$.
\item Let $C\sq\Fix(\mf{F})$ be a connected component (assuming $\Fix(\mf{F})\neq\emptyset$). Choose a point $x\in C$, set $G:=\pi_1(X,x)$ and consider the induced subgroup $\mf{F}_*\leq\Aut(G)$. Then $\Fix(\mf{F}_*)\cong\pi_1(C)$.
\end{enumerate}
\end{lem}
\begin{proof}
We begin with part~(1). Set $G:=\pi_1(X,x)$ for some $x\in X$ and let $G\acts\wt X$ be the action on the universal cover by deck transformations. By assumption, the lifts to $\wt X$ of the elements of $\mf{F}$ extend this action to an isometric action $\wt G\acts\wt X$, for an extension $1\ra G\ra\wt G\ra F\ra 1$ such that the monodromy homomorphism $F\ra\Out(G)$ is precisely the natural embedding of $F\leq\Out(G)$.

If $F$ admits a lift to a finite subgroup $\wt F\leq\Aut(G)$, we can write $\wt G=G\rtimes\wt F$. To avoid confusion, denote by $\wt{\mf{F}}$ the image of $\wt F$ within $\Isom(\wt X)$ given by the action $\wt G\acts\wt X$. Since $\wt X$ is $\CAT$ and complete, the finite group $\wt{\mf{F}}$ fixes a point of $\wt X$, which implies that $\mf{F}$ fixes a point of $X$. The subset $\Fix(\wt{\mf{F}})\sq \wt X$ is convex, hence connected. It follows that it projects to a (whole) connected component of $\Fix(\mf{F})$. Furthermore, replacing $\wt F$ with a $G$--conjugate does not alter the projection to $X$ of $\Fix(\wt{\mf{F}})$. 

Conversely, consider a connected component $C\sq\Fix(\mf{F})$. Given a point $y\in C$ and a lift $\tilde y\in\wt X$, we can consider the subgroup $\wt F_{\tilde y}<\wt G$ of lifts of elements of $\mf{F}$ fixing $\tilde y$. If $\tilde z$ is a lift of another point $z\in C$, the subgroups $\wt F_{\tilde y}$ and $\wt F_{\tilde z}$ are $G$--conjugate.

Summing up, we have shown how to associate a connected component of $\Fix(\mf{F})$ to each $G$--conjugacy class of finite subgroups of $\wt G$ projecting isomorphically onto $F$, and conversely such a conjugacy class to each connected component of $\Fix(\mf{F})$. It is straightforward to check that these two operations are one the inverse of the other, so this completes the proof of part~(1).

We now prove part~(2). Consider again $G=\pi_1(X,x)$ and the action $G\acts\wt X$, where now we have $x\in C\sq\Fix(\mf{F})$. In particular, every $\Phi\in\mf{F}$ induces an automorphism $\Phi_*\colon G\ra G$. We denote by $\mf{F}_*\leq\Aut(G)$ the subgroup consisting of all these automorphisms.

Choose a lift $\tilde x\in\wt X$ of $x$ and let $\wt{\mf{F}}$ be the group of lifts of elements of $\mf{F}$ fixing $\tilde x$. We have $\wt\Phi\o g=\Phi_*(g)\o\wt\Phi$ for every $g\in G$ and every $\Phi\in\mf{F}$ with lift $\wt\Phi\in\wt{\mf{F}}$. This implies that the action of $\wt{\mf{F}}$ on $\wt X$ commutes with that of the subgroup $\Fix(\mf{F}_*)\leq G$. It follows that the subgroup $\Fix(\mf{F}_*)$ preserves the subset $\Fix(\wt{\mf{F}})\sq\wt X$.

Conversely, suppose that some $g\in G$ maps a point $y\in\Fix(\wt{\mf{F}})$ to a point $gy\in\Fix(\wt{\mf{F}})$. Then
\[ gy=\wt\Phi(gy)=\Phi_*(g)\wt\Phi(y)=\Phi_*(g)y, \]
for each $\Phi\in\mf{F}$, which implies that $g\in\Fix(\mf{F}_*)$ since $G$ acts freely on $\wt X$. In conclusion, $\Fix(\mf{F}_*)$ is the $G$--stabiliser of $\Fix(\wt{\mf{F}})\sq\wt X$ and the distinct $G$--translates of $\Fix(\wt{\mf{F}})$ are pairwise disjoint.

Now, $C\sq X$ is locally convex, as it is a connected component of the fixed set of an isometry of $X$ and $X$ is locally uniquely geodesic. Thanks to non-positive curvature, it follows that $C$ is $\pi_1$--injective in $X$ and the lift of $C$ based at $\tilde x$ is an embedded copy $\wt C\sq\wt X$ of the universal cover of $C$. It is clear that $\wt C=\Fix(\wt{\mf{F}})$. We have seen that $\Fix(\mf{F}_*)$ is the $G$--stabiliser of $\wt C$, so it follows that $\pi_1(C)\cong\Fix(\mf{F}_*)$, completing the proof of part~(2).
\end{proof}

\subsection{Automorphisms of RAAGs}

Let $A_{\G}$ be a right-angled Artin group. The vertices of the graph $\G$ form a generating set of $A_{\G}$ and we will refer to them as the \emph{standard generators}. 

Laurence and Servatius \cite{Laurence,Servatius} showed that the automorphism group $\Aut(A_{\G})$ is generated by the following elementary automorphisms.
\begin{itemize}
\item \emph{Graph automorphisms}. Every automorphism of the graph $\G$ gives a permutation of the standard generating set of $A_{\G}$ that defines an automorphism of $A_{\G}$.
\item \emph{Inversions}. For each vertex $v\in\G$, there is an automorphism of $A_{\G}$ that maps $v\mapsto v^{-1}$ and fixes all other standard generators.
\item \emph{Partial conjugations}. For every vertex $v\in\G$ and every connected component $C\sq\G-\St(v)$, there is an automorphism of $A_{\G}$ that maps $x\mapsto vxv^{-1}$ for every $x\in C^{(0)}$ and fixes all other standard generators.
\item \emph{Transvections}. For any two vertices $v,w\in\G$ with $\lk(w)\sq\St(v)$, there is an automorphism of $A_{\G}$ that maps $w\mapsto wv$ and fixes all other standard generators.

We distinguish between \emph{folds} (when $v$ and $w$ are not adjacent in $\G$, or equivalently when $\lk(v)\sq\lk(w)$) and \emph{twists} (when $v$ and $w$ are adjacent in $\G$, or equivalently $\St(v)\sq\St(w)$).
\end{itemize}

An automorphism is \emph{untwisted} \cite{CSV} if it lies in the subgroup $U(A_{\G})\leq\Aut(A_{\G})$ generated by graph automorphisms, inversions, partial conjugations and folds (we are ruling out twists). By \cite{Fio10a}, untwisted automorphisms are precisely those that coarsely preserve the median operator induced on $A_{\G}$ by the universal cover of the Salvetti complex and its cubical structure.

We also define the \emph{pure} untwisted subgroup $U^0(A_{\G})\leq U(A_{\G})$ as the finite-index subgroup generated by inversions, partial conjugations and folds (we are now also ruling out graph automorphisms). Note that, although graph automorphisms are not allowed as generators of $U^0(A_{\G})$, some of them can still lie end up lying in $U^0(A_{\G})$ (see e.g.\ \cite[Proposition~3.3]{Day-Wade}).

Inner automorphisms lie in $U^0(A_{\G})$, as they are products of partial conjugations. Thus, whether an automorphism $\varphi\in\Aut(A_{\G})$ belongs to $U^0(A_{\G})$ or not depends only on its projection $[\varphi]\in\Out(A_{\G})$. The same applies to $U(A_{\G})$.

\section{Fixed subgroups of finite-order untwisted automorphisms}\label{sect:meat}

This section is devoted to the proof of Theorem~\ref{thmintro:divisible_vs_fix} and Corollary~\ref{corintro:examples}. 

In Subsection~\ref{subsec:divisible}, we define divisible cube complexes and provide a few examples. In particular, Examples~\ref{ex:strongly_divisible} and~\ref{ex:braid_divisible} show that the groups appearing in Corollary~\ref{corintro:examples} are fundamental groups of such cube complexes.

In Subsection~\ref{subsec:host}, we show that every divisible cube complex $Q$ can be \emph{embedded} (rather than simply immersed) into a ``host'': a particularly nice cube complex whose fundamental group is a right-angled Artin group, and which resembles one of the Salvetti blowups of \cite{CSV} (though is not always one). Then, in Subsection~\ref{subsec:extended_host}, we enlarge the host so that it supports a cubical automorphisms having the subcomplex $Q$ as a full connected component of its fixed set. This proves one direction of the equivalence in Theorem~\ref{thmintro:divisible_vs_fix}, namely Theorem~\ref{thm:divisible->fix}.

Finally, in Subsection~\ref{subsec:converse} we prove the other half of Theorem~\ref{thmintro:divisible_vs_fix}, see Proposition~\ref{prop:fix->divisible}.

\subsection{Divisible cube complexes}\label{subsec:divisible}

Let $Q$ be a compact special cube complex. 

Before introducing divisible cube complexes, it is useful to define the following simpler and slightly stronger concept, which is sufficient to prove most of Corollary~\ref{corintro:examples} (see Example~\ref{ex:strongly_divisible}). 

\begin{defn}[Strongly divisible]\label{defn:strongly_divisible}
We say that $Q$ is \emph{strongly divisible} if \mbox{both of the following hold.}
\begin{enumerate}
\setlength\itemsep{.25em}
\item For any two distinct vertices $x,y\in Q^{(0)}$, there exist pairwise-disjoint hyperplanes $H_1,\dots,H_k$ such that $x$ and $y$ lie in distinct connected components of $Q-(H_1\cup\dots\cup H_k)$.
\item For any vertex $x\in Q^{(0)}$ and any hyperplane half-carrier $C\sq Q$ such that $x\not\in C$, there exist pairwise-disjoint hyperplanes $H_1,\dots,H_k$ such that $x$ and the half-carrier $C$ are contained in distinct connected components of $Q-(H_1\cup\dots\cup H_k)$.
\end{enumerate}
\end{defn}

A ``divisible'' cube complex is similar to a strongly divisible one, except that there might be certain sets of hyperplanes that we cannot distinguish from each other, as given by a special colouring.

If $(\kappa,\G)$ is a special colouring for $Q$ (as in Definition~\ref{defn:special_colouring}), we denote by $C(v)\sq Q$ the union of the carriers of the hyperplanes in the preimage $\kappa^{-1}(v)\sq\mscr{W}(Q)$. Similarly, for $\eps\in\{\pm\}$, we denote by $C(v^{\eps})$ the union of the positive carriers of the oriented hyperplanes in $\kappa^{-1}(v^{\eps})\sq\mscr{O}(Q)$. We still refer to $C(v)$ as a (generalised) \emph{carrier}, and to $C(v^-)$ and $C(v^+)$ as (generalised) \emph{half-carriers}.

\begin{defn}[Dividing pattern]\label{defn:dividing_pattern}
A \emph{dividing pattern} for $Q$ is the data of a collection $\mc{H}_1,\dots,\mc{H}_n$ of partitions $Q=\mc{H}_i^-\sqcup\mc{H}_i^0\sqcup\mc{H}_i^+$ and a minimal special colouring $(\kappa,\G)$ \mbox{such that the following hold.}
\begin{enumerate}
\setlength\itemsep{.25em}
\item Each set $\mc{H}_i^0$ is a union of pairwise-disjoint hyperplanes of $Q$. If $\mc{H}_i^0$ contains a hyperplane $H\in\mscr{W}(Q)$, then it also contains all other hyperplanes $K\in\mscr{W}(Q)$ with $\kappa(H)=\kappa(K)$.
\item Each of the sets $\mc{H}_i^{\pm}$ is a union of connected components of $Q-\mc{H}_i^0$.
\item If $\mc{H}_i^0$ contains the hyperplanes in some preimage $\kappa^{-1}(v)\sq\mscr{W}(Q)$, then either $C(v^+)\sq\mc{H}_i^+$ and $C(v^-)\sq\mc{H}_i^-$, or $C(v^+)\sq\mc{H}_i^-$ and $C(v^-)\sq\mc{H}_i^+$.
\item The partitions $\mc{H}_i$ suffice to separate vertices of $Q$ from each other and from half-carriers. That is:
	\begin{itemize}
	\item if $x,y\in Q^{(0)}$ are distinct, then $x\in\mc{H}_i^-$ and $y\in\mc{H}_i^+$ (or vice versa) for at least one $i$;
	\item if $x\in Q^{(0)}$ lies outside the half-carrier $C(v^{\eps})$ for some $v\in\G$ and $\eps\in\{\pm\}$, then there exist $i$ and $\eta\in\{\pm\}$ such that $x\in\mc{H}_i^{\eta}$ and $C(v^{\eps})\sq\mc{H}_i^{-\eta}$. 
	\end{itemize}
\end{enumerate}
\end{defn}

\begin{defn}[Divisible]\label{defn:divisible}
We say that $Q$ is \emph{divisible} if it admits a dividing pattern.
\end{defn}

The following easy observation shows how strongly divisible cube complexes are divisible. Recall that standard special colourings are those for which the map $\kappa\colon\mscr{W}(Q)\ra\G^{(0)}$ is a bijection.

\begin{lem}
A compact special cube complex $Q$ is strongly divisible if and only if it admits a dividing pattern with respect to a standard special colouring.
\end{lem}
\begin{proof}
If $Q$ admits a dividing pattern in which all the preimages $\kappa^{-1}(v)\sq\mscr{W}(Q)$ are singletons, then it is clear that $Q$ is strongly divisible. 

Conversely, suppose that $Q$ is strongly divisible and we want to construct a dividing pattern with respect to a standard special colouring $\kappa\colon\mscr{W}(Q)\ra\G_Q$. Consider the set $\mscr{H}$ of all pairs $(\mf{H},\mf{C})$, where $\mf{H}\sq Q$ is a union of pairwise disjoint hyperplanes, and $\mf{C}$ is a connected component of $Q-\mf{H}$. Let $\mscr{H}^*\sq\mscr{H}$ be the subset of pairs $(\mf{H},\mf{C})$ with the additional property that $\mf{C}$ intersects exactly one half-carrier of each hyperplane contained in $\mf{H}$; in other words, no hyperplane in $\mf{H}$ has the $0$--skeleton of its carrier entirely contained in $\mf{C}$, and no hyperplane in $\mf{H}$ is separated from $\mf{C}$ by the other hyperplanes in $\mf{H}$.

Each element $\pi=(\mf{H},\mf{C})\in\mscr{H}$ determines a partition $\mc{H}_{\pi}$ with $\mc{H}_{\pi}^0=\mf{H}$, $\mc{H}_{\pi}^+=\mf{C}$ and $\mc{H}_{\pi}^-=Q-(\mf{H}\cup\mf{C})$. Each such partition satisfies conditions~(1) and~(2) in Definition~\ref{defn:dividing_pattern}, since the special colouring $\kappa$ is assumed to be injective. If $\pi\in\mscr{H}^*$, then the partition $\mc{H}_{\pi}$ also satisfies condition~(3) of the same definition (in particular, $\mc{H}_{\pi}^0$ and $\mc{H}_{\pi}^{\pm}$ are all nonempty).

Finally, since $Q$ is strongly divisible, the collection of partitions $\{\mc{H}_{\pi}\mid\pi\in\mscr{H}\}$ clearly satisfies condition~(4) of Definition~\ref{defn:dividing_pattern}. Note that, if two subsets $A,B\sq Q$ are separated by a partition $\mc{H}_{\pi}$ with $\pi\in\mscr{H}$, in the sense that $A\sq\mc{H}_{\pi}^+$ and $B\sq\mc{H}_{\pi}^-$ or vice versa, then $A,B$ are also separated by some $\mc{H}_{\s}$ with $\s\in\mscr{H}^*$. Indeed, if $\pi=(\mf{H},\mf{C})$, let $\mf{H}'\sq\mf{H}$ be the subset formed by hyperplanes of which $\mf{C}$ intersects exactly one half-carrier, and let $\mf{C}'\supseteq\mf{C}$ be the set obtained by adding to $\mf{C}$ all hyperplanes contained in $\mf{H}-\mf{H}'$ and such that all vertices of their carrier lie in $\mf{C}$. Then, setting $\s:=(\mf{H}',\mf{C}')$, we have $\s\in\mscr{H}^*$ and any two subsets separated by $\mc{H}_{\pi}$ are also separated by $\mc{H}_{\s}$, since we have both $\mc{H}_{\pi}^-\sq\mc{H}_{\s}^-$ and $\mc{H}_{\pi}^+\sq\mc{H}_{\s}^+$.

In conclusion, the collection of partitions $\{\mc{H}_{\pi}\mid\pi\in\mscr{H}^*\}$ satisfies all four conditions from Definition~\ref{defn:dividing_pattern}. (Of course, this is a much larger dividing pattern than necessary, and it has the feature that all positive sides $\mc{H}_{\pi}^+$ are connected, which is also not strictly necessary.)
%
\end{proof}

Note that condition~(3) in Definition~\ref{defn:dividing_pattern} is not much of a requirement when $(\kappa,\G)$ is standard, but it actually places rather strong restrictions when the preimages $\kappa^{-1}(v)$ are not singletons. Every hyperplane in $\mc{H}_i^0$ must be actively involved in separating $\mc{H}_i^-$ from $\mc{H}_i^+$, but at the same time $\mc{H}_i^0$ is not allowed to only contain \emph{some} of the hyperplanes in a given $\kappa^{-1}(v)$, because of Definition~\ref{defn:dividing_pattern}(1).

The sets $\kappa^{-1}(v)\sq\mscr{W}(Q)$ and $\mc{H}_i^0\sq Q$ appearing in Definition~\ref{defn:dividing_pattern} might appear superficially similar, but we emphasise that they are subject to rather different requirements. For one thing, we have $\kappa^{-1}(v)\cap\kappa^{-1}(v')=\emptyset$ for $v\neq v'$, whereas $\mc{H}_i^0$ and $\mc{H}_{i'}^0$ can share hyperplanes for $i\neq i'$. In addition, the hyperplanes in a given $\kappa^{-1}(v)$ must have pairwise disjoint (oriented) carriers, whereas the hyperplanes appearing in a given $\mc{H}_i^0$ are simply required to be pairwise disjoint. We also allow ``inter-osculations'' between the $\mc{H}_i^0$, and between some $\mc{H}_i^0$ and some $\kappa^{-1}(v)$.

The following two examples show how to deduce Corollary~\ref{corintro:examples} from Theorem~\ref{thmintro:divisible_vs_fix}.

\begin{ex}[Strongly divisible examples]\label{ex:strongly_divisible}
The following groups can be realised as fundamental groups of strongly divisible compact special cube complexes.
\begin{enumerate}
\setlength\itemsep{.25em}
\item \emph{The fundamental group of any closed orientable surface $S_g$.}

This is clear for $g=0,1$. When $g\geq 2$, consider the square complex $Q_g$ dual to the family of curves on $S_g$ depicted in Figure~\ref{fig:surface_1}. The latter consists of $3g-3$ blue curves forming a pants decomposition of $S_g$, and $g+1$ additional red curves that cut each pant into a ``top'' hexagon and a ``bottom'' one. 

It is clear that $\pi_1(Q_g)\cong\pi_1(S_g)$, that $Q_g$ is special and that it has $4g-4$ vertices, one for each hexagon. For any hexagon $X\sq S_g$, consider the three blue curves that bound the pant $P$ of which $X$ is part: they separate $X$ from any hexagon outside $P$, as well as from every blue/red curve other than the six bounding $X$. In addition, the top and bottom hexagons of each pant are separated by the red curves. This shows that $Q_g$ is strongly divisible.

\begin{figure}[ht]
\begin{tikzpicture}
\draw[thick] (0,0) ellipse (5.6cm and 2.5cm);
\draw[thick] (-4.5,0) ellipse (0.5cm and 0.5cm);
\draw[thick] (-3,0) ellipse (0.5cm and 0.5cm);
\draw[thick] (-1.5,0) ellipse (0.5cm and 0.5cm);
\draw[thick] (1.5,0) ellipse (0.5cm and 0.5cm);
\draw[thick] (3,0) ellipse (0.5cm and 0.5cm);
\draw[thick] (4.5,0) ellipse (0.5cm and 0.5cm);
\node at (0,0) {$\dots$};
\tikzset{
    partial ellipse/.style args={#1:#2:#3}{
        insert path={+ (#1:#3) arc (#1:#2:#3)}
    }
}
\draw[ultra thick, blue] (-5.3,0) [partial ellipse=0:180:0.3cm and 0.15cm];
\draw[ultra thick, blue, dashed] (-5.3,0) [partial ellipse=180:360:0.3cm and 0.15cm];
\draw[ultra thick, blue] (-3.75,0) [partial ellipse=0:180:0.25cm and 0.15cm];
\draw[ultra thick, blue, dashed] (-3.75,0) [partial ellipse=180:360:0.25cm and 0.15cm];
\draw[ultra thick, blue] (-2.25,0) [partial ellipse=0:180:0.25cm and 0.15cm];
\draw[ultra thick, blue, dashed] (-2.25,0) [partial ellipse=180:360:0.25cm and 0.15cm];
\draw[ultra thick, blue] (2.25,0) [partial ellipse=0:180:0.25cm and 0.15cm];
\draw[ultra thick, blue, dashed] (2.25,0) [partial ellipse=180:360:0.25cm and 0.15cm];
\draw[ultra thick, blue] (3.75,0) [partial ellipse=0:180:0.25cm and 0.15cm];
\draw[ultra thick, blue, dashed] (3.75,0) [partial ellipse=180:360:0.25cm and 0.15cm];
\draw[ultra thick, blue] (5.3,0) [partial ellipse=0:180:0.3cm and 0.15cm];
\draw[ultra thick, blue, dashed] (5.3,0) [partial ellipse=180:360:0.3cm and 0.15cm];
\draw[ultra thick, blue, dashed] (-3,1.3) [partial ellipse=-90:90:0.25cm and 0.8cm];
\draw[ultra thick, blue] (-3,1.3) [partial ellipse=90:270:0.25cm and 0.8cm];
\draw[ultra thick, blue, dashed] (-1.5,1.45) [partial ellipse=-90:90:0.3cm and 0.95cm];
\draw[ultra thick, blue] (-1.5,1.45) [partial ellipse=90:270:0.3cm and 0.95cm];
\draw[ultra thick, blue] (3,1.3) [partial ellipse=-90:90:0.25cm and 0.8cm];
\draw[ultra thick, blue, dashed] (3,1.3) [partial ellipse=90:270:0.25cm and 0.8cm];
\draw[ultra thick, blue] (1.5,1.45) [partial ellipse=-90:90:0.3cm and 0.95cm];
\draw[ultra thick, blue, dashed] (1.5,1.45) [partial ellipse=90:270:0.3cm and 0.95cm];
\draw[ultra thick, blue, dashed] (-3,-1.3) [partial ellipse=-90:90:0.25cm and 0.8cm];
\draw[ultra thick, blue] (-3,-1.3) [partial ellipse=90:270:0.25cm and 0.8cm];
\draw[ultra thick, blue, dashed] (-1.5,-1.45) [partial ellipse=-90:90:0.3cm and 0.95cm];
\draw[ultra thick, blue] (-1.5,-1.45) [partial ellipse=90:270:0.3cm and 0.95cm];
\draw[ultra thick, blue] (3,-1.3) [partial ellipse=-90:90:0.25cm and 0.8cm];
\draw[ultra thick, blue, dashed] (3,-1.3) [partial ellipse=90:270:0.25cm and 0.8cm];
\draw[ultra thick, blue] (1.5,-1.45) [partial ellipse=-90:90:0.3cm and 0.95cm];
\draw[ultra thick, blue, dashed] (1.5,-1.45) [partial ellipse=90:270:0.3cm and 0.95cm];
\draw[very thick, red] (0,0) ellipse (5.43cm and 2.33cm);
\draw[very thick, red] (-4.5,0) ellipse (0.67cm and 0.67cm);
\draw[very thick, red] (-3,0) ellipse (0.67cm and 0.67cm);
\draw[very thick, red] (-1.5,0) ellipse (0.67cm and 0.67cm);
\draw[very thick, red] (1.5,0) ellipse (0.67cm and 0.67cm);
\draw[very thick, red] (3,0) ellipse (0.67cm and 0.67cm);
\draw[very thick, red] (4.5,0) ellipse (0.67cm and 0.67cm);
\end{tikzpicture}
\caption{A family of curves on $S_g$ whose dual square complex is strongly divisible.}
\label{fig:surface_1}
\end{figure}
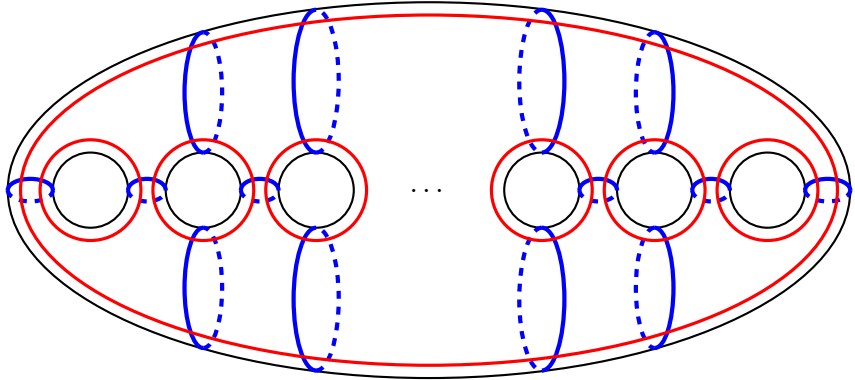

\item \emph{The commutator subgroup $W_{\G}'$ of any right-angled Coxeter group $W_{\G}$.} 

A special cube complex $Q_{\G}$ with fundamental group $W_{\G}'$ was described in \cite{Droms}; it is simply the quotient of the Davis complex for $W_{\G}$ (introduced in \cite{Davis-Annals,Davis-book}) by the action of $W_{\G}'$ . To construct $Q_{\G}$, start with the cube $[0,1]^V$ with $V=\G^{(0)}$ and label each of its edges by the corresponding vertex of $\G$. The cube complex $Q_{\G}$ is the subcomplex of $[0,1]^V$ formed by those cubes whose edges are labelled by the vertices of a clique of $\G$. 

Since $(\Z/2\Z)^V$ acts vertex-transitively on $Q_{\G}$ by flipping the intervals $[0,1]$, it suffices to consider the vertex $\underline{0}=(0,\dots,0)\in Q_{\G}$ in order to prove that $Q_{\G}$ is strongly divisible.

For each $v\in V$, let $\mc{H}_v^0\sq Q_{\G}$ be the union of all hyperplanes of $Q_{\G}$ dual to edges labelled by $v$. The complement of $\mc{H}_v^0$ has only two connected components, one containing the vertex $\underline{0}=(0,\dots,0)$ (call it $\mc{H}_v^-$) and one containing the vertex $\underline{1}=(1,\dots,1)$ (call it $\mc{H}_v^+$). More precisely, a vertex of $Q_{\G}$ lies in $\mc{H}_v^-$ if its coordinate in position $v$ is $0$, and it lies in $\mc{H}_v^+$ if this coordinate is $1$.

Note that $\bigcap_{v\in V}\mc{H}_v^-=\{\underline{0}\}$, so $Q_{\G}$ satisfies part~(1) of Definition~\ref{defn:strongly_divisible}. Regarding part~(2), consider a hyperplane $H\sq Q_{\G}$ such that $\underline{0}\not\in C(H)$; say that $H$ is labelled by some $v\in V$. If some $w\in V-\{v\}$ is not connected to $v$ by an edge of $\G$, then all vertices of $C(H)$ have the same coordinate $\eps_w\in\{0,1\}$ in position $w$. Since $\underline{0}\not\in C(H)$, there must be at least one such $w\in V$ with $\eps_w=1$ (as all other coordinates vary freely). It follows that $C(H)\sq\mc{H}_w^+$, while $\underline{0}\in\mc{H}_w^-$. In conclusion, the cube complex $Q_{\G}$ is strongly divisible.
\end{enumerate}
\end{ex}

\begin{ex}[Graph braid groups are divisible]\label{ex:braid_divisible}
Let $\G$ be a connected, finite, simplicial graph and consider an integer\footnote{The unmarked configuration space $UC_n(\G)$ described in the next paragraph makes sense for all $0\leq n\leq\#\G^{(0)}$. However, setting $v:=\#\G^{(0)}$, we have $UC_0(\G)\cong UC_v(\G)\cong\{\ast\}$ and $UC_1(\G)\cong UC_{v-1}(\G)\cong\G$, so its fundamental group is never particularly interesting in these cases, justifying our stronger restrictions on the integer $n$.} $2\leq n\leq\#\G^{(0)}-2$. The graph braid group $B_n(\G)$ is the fundamental group of a divisible, compact, special cube complex.

To begin with, $B_n(\G)$ is defined as the fundamental group of the unmarked configuration space $UC_n(\G)$. This is a connected, compact, special cube complex having a vertex for every cardinality--$n$ subset $S\sq\G^{(0)}$. Two such subsets $S,S'$ are joined by an edge $E\sq UC_n(\G)$ if their symmetric difference $S\triangle S'$ is $\{x,x'\}$ where $x,x'$ are the vertices of an edge $e\sq\G$; we will say that $E$ is \emph{subordinate} to $e$. There is a $k$--cube in $UC_n(\G)$ for any $k$ pairwise-disjoint edges $e_1,\dots,e_k\sq\G$ and any cardinality--$(n-k)$ subset $S_0\sq\G-(e_1\cup\dots\cup e_k)$: the vertices of this cube are the $2^k$ subsets of $\G^{(0)}$ of the form $S_0\cup\{x_1^{\eps_1},\dots,x_k^{\eps_k}\}$, where $\eps_i\in\{\pm\}$ and $x_i^-,x_i^+$ are the vertices of $e_i$ (see e.g.\ \cite[Section~3]{Gen-IJAC} for more details).

Let us show that the cube complex $Q:=UC_n(\G)$ is divisible. 

First note that, for every hyperplane $H\in\mscr{W}(Q)$, all edges of $Q$ dual to $H$ are subordinate to the same edge $e\sq\G$; thus, we will also say that $H$ is subordinate to $e$. In general, it is possible for distinct hyperplanes of $Q$ to be subordinate to the same edge of $\G$ (see \cite[Lemma~3.6]{Gen-IJAC}). If $e\sq\G$ is an edge with vertices $x,x'$, then a vertex $S\in Q$ lies in the carrier of a hyperplane $H$ subordinate to $e$ if and only if $\#S\cap\{x,x'\}=1$; in this case, the two half-carriers of $H$ corresponds to vertices $S\in C(H)\sq Q$ with $S\cap\{x,x'\}=\{x\}$ and $S\cap\{x,x'\}=\{x'\}$, respectively.

We can define a (minimal) special colouring $\kappa\colon\mscr{W}(Q)\ra\mc{E}(\G)$, where $\mc{E}(\G)$ is the graph having a vertex for every edge of $\G$, with two edges of $\G$ connected by an edge of $\mc{E}(\G)$ exactly when then are disjoint. The map $\kappa$ simply assigns to each hyperplane of $Q$ the edge of $\G$ that it is subordinate to. In order to define the map $\kappa\colon\mscr{O}(Q)\ra\mc{E}(\G)^{\pm}$, it is convenient to think of $\mc{E}(\G)^{\pm}$ as the set of pairs $(e,x)$ where $e\sq\G$ is an edge and $x$ is one of the two vertices of $e$. If $H\in\mscr{W}(Q)$ is a hyperplane subordinate to $e$ and if $\ora{H}$ is an orientation on $H$, then all vertices $S$ of the positive carrier $C(\ora{H})\sq Q$ are subsets of $\G$ containing the same vertex of $e$; calling $x$ this vertex, we define $\kappa(\ora{H})=(e,x)$. This gives the map $\kappa\colon\mscr{O}(Q)\ra\mc{E}(\G)^{\pm}$ that we will consider in our special colouring (which is identical to the one in \cite[Proposition~3.7]{Gen-IJAC}).

Now, we can construct a dividing pattern with respect to this special colouring. For each vertex $a\in\G$, there is a partition $Q=\mc{H}_a^-\sqcup\mc{H}_a^0\sqcup\mc{H}_a^+$ defined as follows. The set $\mc{H}_a^0\sq Q$ is the union of all hyperplanes of $Q$ subordinate to edges of $\G$ incident to $a$; the latter edges pairwise intersect, so the hyperplanes in $\mc{H}_a^0$ are pairwise disjoint. The set $\mc{H}_a^+$ is the union of all connected components of $Q-\mc{H}_a^0$ containing vertices $S\in Q$ with $a\in S$; the set $\mc{H}_a^-$ is defined similarly in terms of vertices $S\in Q$ with $a\not\in S$. It is clear that the family of partitions $\{\mc{H}_a \mid a\in\G^{(0)}\}$ and the special colouring $(\kappa,\mc{E}(\G))$ together satisfy conditions~(1), (2) and~(3) in Definition~\ref{defn:dividing_pattern}.

We are only left to check condition~(4). If vertices $S,S'\in Q$ are distinct, then there exists a vertex $a\in\G$ such that $a\in S$ and $a\not\in S'$. In other words $S\in\mc{H}_a^+$ and $S'\in\mc{H}_a^-$, showing that the first half of Definition~\ref{defn:dividing_pattern}(4) is satisfied. For the second half, consider a vertex $S\in Q$ and an oriented hyperplane $\ora{H}\in\mscr{O}(Q)$ with $\kappa(\ora{H})=(e,x)$. Note that the generalised half-carrier $C(e,x)$ (the union of the positive carriers of the oriented hyperplanes in $\kappa^{-1}(e,x)$) contains $S$ if and only if $S\cap e=\{x\}$. If $S\not\in C(e,x)$, then either $x\not\in S$ (in which case $S\in\mc{H}_x^-$ and $C(x,e)\sq\mc{H}_x^+$) or we have $e\sq S$ (in which case, denoting by $x'$ the vertex of $e$ other than $x$, we have $S\in\mc{H}_{x'}^+$ and $C(x,e)\sq\mc{H}_{x'}^-$). Either way Definition~\ref{defn:dividing_pattern}(4) is satisfied, proving that $Q=UC_n(\G)$ is divisible.

Note that the specific integer $n$ does not play any explicit role in the definition of the above dividing pattern for $UC_n(\G)$. This will have interesting consequences in Example~\ref{ex:explicit}(3) below.
\end{ex}

Next, we observe that strongly divisible and divisible cube complexes are indeed different classes, and compare them to possible weakenings of their definition.

\begin{figure}[ht]
\begin{tikzpicture}[baseline=(current bounding box.center)]
\draw[fill] (-1,0) circle [radius=0.07cm];
\draw[fill] (0,1) circle [radius=0.07cm];
\draw[fill] (1,0) circle [radius=0.07cm];
\draw[fill] (0,-1) circle [radius=0.07cm];
\draw[fill,thick] (-1,0) -- (0,1);
\draw[fill,thick] (0,1) -- (1,0);
\draw[fill,thick] (1,0) -- (0,-1);
\draw[fill,thick] (0,-1) -- (-1,0);
\path[fill=black,opacity=0.1] (-1,0) -- (0,1) -- (1,0) -- (0,-1) -- cycle;
\node [above] at (-1,0) {$x$};
\filldraw[black,opacity=0.1,draw=black] (0,-1) plot [smooth] coordinates {(0,-1) (-1,-0.5) (-1.25,0) (-1,0.5) (0,1)} --  (-2,1) -- (-2,-1) -- (0,-1);
\draw[thick] plot [smooth] coordinates {(0,-1) (-1,-0.5) (-1.25,0) (-1,0.5) (0,1)};
\draw[fill,thick] (0,-1) -- (-2,-1);
\draw[fill,thick] (-2,-1) -- (-2,1);
\draw[fill,thick] (0,1) -- (-2,1);
\draw[fill] (-2,-1) circle [radius=0.07cm];
\draw[fill] (-2,1) circle [radius=0.07cm];
\node [above] at (-1,1) {$e$};
\node [left] at (-2,0) {$e'$};
\filldraw[black,opacity=0.1,draw=black] (-1,0) plot [smooth] coordinates {(-1,0) (-0.88,-0.35) (-0.5,-1) (-0.25,-1.2) (0,-1.25) (0.5,-1) (1,0)} --  (1,-2) -- (-1,-2) -- (-1,0);
\draw[thick] plot [smooth] coordinates {(-1,0) (-0.88,-0.35)};
\draw[dashed,thick] plot [smooth] coordinates {(-0.88,-0.35) (-0.5,-1) (-0.25,-1.2)};
\draw[thick] plot [smooth] coordinates {(-0.25,-1.2) (0,-1.25) (0.5,-1) (1,0)};
\draw[fill,thick] (1,-2) -- (1,0);
\draw[fill,thick] (-1,-2) -- (1,-2);
\draw[fill,thick] (-1,-2) -- (-1,-1.25);
\draw[dashed,thick] (-1,-1.25) -- (-1,-0.35);
\draw[fill,thick] (-1,-0.35) -- (-1,0);
\draw[fill] (1,-2) circle [radius=0.07cm];
\draw[fill] (-1,-2) circle [radius=0.07cm];
\node [right] at (1,-1) {$f$};
\node [below] at (0,-2) {$f'$};
\end{tikzpicture}
\hspace{1.5cm}
\begin{tikzpicture}[baseline=(current bounding box.center)]
\draw[fill] (-1,0) circle [radius=0.07cm];
\draw[fill] (0,1) circle [radius=0.07cm];
\draw[fill] (1,0) circle [radius=0.07cm];
\draw[fill] (0,-1) circle [radius=0.07cm];
\draw[fill,thick] (-1,0) -- (0,1);
\draw[fill,thick] (0,1) -- (1,0);
\draw[fill,thick] (1,0) -- (0,-1);
\draw[fill,thick] (0,-1) -- (-1,0);
\path[fill=black,opacity=0.1] (-1,0) -- (0,1) -- (1,0) -- (0,-1) -- cycle;
\node [above] at (-1,0) {$x$};
\filldraw[black,opacity=0.1,draw=black] (0,-1) plot [smooth] coordinates {(0,-1) (-1,-0.5) (-1.25,0) (-1,0.5) (0,1)} --  (-2,1) -- (-2,-1) -- (0,-1);
\draw[thick] plot [smooth] coordinates {(0,-1) (-1,-0.5) (-1.25,0) (-1,0.5) (0,1)};
\draw[fill,thick] (0,-1) -- (-2,-1);
\draw[fill,thick] (-2,-1) -- (-2,1);
\draw[fill,thick] (0,1) -- (-2,1);
\draw[fill] (-2,-1) circle [radius=0.07cm];
\draw[fill] (-2,1) circle [radius=0.07cm];
\node [above] at (-1,1) {$e$};
\node [below] at (-1,-1) {\textcolor{white}{$e$}};
\end{tikzpicture}
\hspace{1.5cm}
\begin{tikzpicture}[baseline=(current bounding box.center)]
\draw[fill,thick,decoration={markings, mark=at position 0.5 with {\arrow{>>}}},
        postaction={decorate}] (-2,0) -- (0,0);
        \draw[fill,thick,decoration={markings, mark=at position 0.5 with {\arrow{>>>}}},
        postaction={decorate}] (0,0) -- (2,0);
\draw[fill,thick,decoration={markings, mark=at position 0.5 with {\arrow{>>>}}},
        postaction={decorate}] (-2,2) -- (0,2);
        \draw[fill,thick,decoration={markings, mark=at position 0.5 with {\arrow{>>}}},
        postaction={decorate}] (0,2) -- (2,2);
\draw[fill,thick] (0,0) -- (0,2);
\draw[fill,thick,decoration={markings, mark=at position 0.5 with {\arrow{>}}},
        postaction={decorate}] (-2,0) -- (-2,2);
\draw[fill,thick,decoration={markings, mark=at position 0.5 with {\arrow{>}}},
        postaction={decorate}] (2,0) -- (2,2);
\draw[fill] (-2,0) circle [radius=0.07cm];
\draw[fill] (0,2) circle [radius=0.07cm];
\draw[fill] (2,0) circle [radius=0.07cm];
\draw[fill,thick,white] (-2,2) circle [radius=0.07cm];
\draw[thick] (-2,2) circle [radius=0.07cm];
\draw[fill,thick,white] (0,0) circle [radius=0.07cm];
\draw[thick] (0,0) circle [radius=0.07cm];
\draw[fill,thick,white] (2,2) circle [radius=0.07cm];
\draw[thick] (2,2) circle [radius=0.07cm];
\path[fill,thick=black,opacity=0.1] (-2,0) -- (-2,2) -- (2,2) -- (2,0) -- cycle;
\node [above] at (0,2) {$x$};
\node [below] at (0,0) {$y$};
\end{tikzpicture}
\caption{Assorted counterexamples.}
\label{fig:non-examples}
\end{figure}
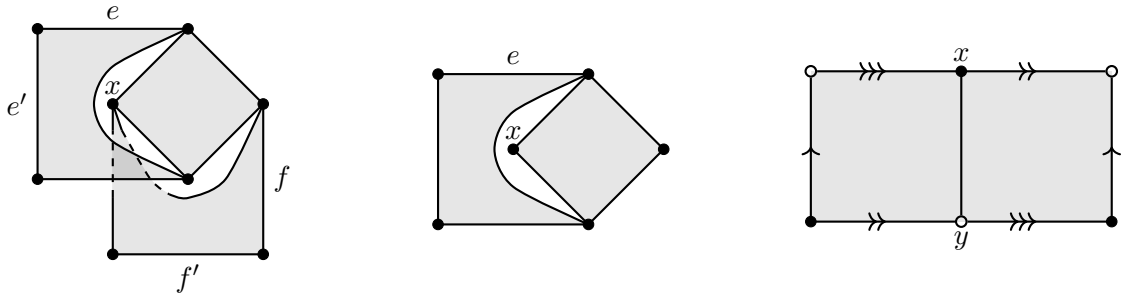

\begin{ex}\label{ex:non-examples}
The various properties of cube complexes considered in this subsection are related by the following straightforward implications:
\[ \text{strongly divisible}\ \ \Ra\ \ \text{divisible}\ \ \Ra\ \ \text{special + Definition~\ref{defn:strongly_divisible}(1)} \ \Ra\ \ \text{special}. \]
The examples in Figure~\ref{fig:non-examples} show that none of these arrows are reversible.

On the left: a divisible cube complex that is not strongly divisible. It consists of three squares sharing vertices as depicted. Let $(\kappa,\G)$ be the (minimal) special colouring in which the hyperplanes dual to the edges $e$ and $f$ form one fibre of $\kappa$, and the hyperplanes dual to $e'$ and $f'$ form another one. The graph $\G$ has $4$ vertices and $2$ disjoint edges. It is straightforward to construct a dividing pattern for the cube complex with respect to $(\kappa,\G)$. On the other hand, the cube complex is not strongly divisible: there is no way of separating the vertex $x$ from the carrier of the hyperplane dual to $e$ via a collection of pairwise-disjoint hyperplanes.

In the middle: a special cube complex that satisfies condition~(1) in Definition~\ref{defn:strongly_divisible} but fails to be divisible. Indeed, the vertex $x$ cannot be separated from the carrier of hyperplane dual to the edge $e$ using pairwise-disjoint hyperplanes, much as in the previous example. Here, on the other hand, there are no special colourings in which the hyperplane dual to $e$ is not alone in its fibre of $\kappa$.

On the right: a special cube complex whose two vertices $x,y$ cannot be separated by pairwise-disjoint hyperplanes. Indeed, there are only two hyperplanes and they are transverse. Note that this cube complex is a double cover of the standard $2$--torus, so this also shows that divisibility is not preserved by passing to finite covers. This example was suggested to me by Sam Shepherd.
\end{ex}

Finally, we conclude this subsection with a few important definitions.

We say that partitions $\mc{H}_i$ and $\mc{H}_j$ are \emph{transverse} if $\mc{H}_i^{\eps_i}\cap\mc{H}_j^{\eps_j}\neq\emptyset$ for all four possible choices of $\eps_i,\eps_j\in\{\pm\}$. Note that it is possible for $\mc{H}_i$ and $\mc{H}_j$ to be transverse even if $\mc{H}_i^0$ and $\mc{H}_j^0$ are disjoint.

Similarly, in the presence of a special colouring $(\kappa,\G)$, we say that $\mc{H}_i$ and a colour $v\in\G$ are \emph{transverse} if each of the generalised half-carriers $C(v^-)$ and $C(v^+)$ intersects both $\mc{H}_i^-$ and $\mc{H}_i^+$. This can occur in two ways: either a hyperplane in $\kappa^{-1}(v)$ crosses a hyperplane in $\mc{H}_i^0$, or there exist two hyperplanes $H,H'\in\kappa^{-1}(v)$ with $C(H)\sq\mc{H}_i^+$ and $C(H')\sq\mc{H}_i^-$.

\begin{defn}[Extended crossing graph]\label{defn:extended_crossing_graph}
Consider a dividing pattern formed by partitions $\mc{H}_1,\dots,\mc{H}_n$ and a special colouring $(\kappa,\G)$. The \emph{extended crossing graph} of the dividing pattern is the smallest graph $\wh\G$ with vertex set $\wh\G^{(0)}=\G^{(0)}\sqcup\{\mc{H}_1,\dots,\mc{H}_n\}$ and edges ensuring that $\G\sq\wh\G$ is a full subgraph, as well as connecting pairs $[\mc{H}_i,\mc{H}_j]$ and $[\mc{H}_i,v]$ that are transverse (in the sense defined above).
\end{defn}

\subsection{The host}\label{subsec:host}

By definition, a special cube complex $Q$ admits a locally isometric immersion into the Salvetti complex of a RAAG. Under the stronger assumption that $Q$ is divisible, we can promote this to a locally convex \emph{embedding} of $Q$ into a nice cube complex $M$, whose fundamental group is still a RAAG and which we term a ``host''. This subsection is devoted to the construction of hosts, which is canonical once a dividing pattern is fixed. The entire procedure draws much of its inspiration from the construction of Salvetti blowups by Charney, Stambaugh and Vogtmann in \cite{CSV}.

Let $Q$ be a divisible compact special cube complex. Consider a dividing pattern formed by partitions $\mc{H}_1,\dots,\mc{H}_n$ and a minimal special colouring $(\kappa,\G)$. Let $\wh\G$ be the extended crossing graph of this dividing pattern. We also think of the colouring $(\kappa,\G)$ as a labelling of edge germs in $Q$ by elements of $\G^{\pm}$, as discussed in Subsection~\ref{subsec:special}.

Throughout the coming discussion, it can be helpful to keep in mind the case when $Q$ is strongly divisible, $\G$ is simply the crossing graph $\G_Q$ of the hyperplanes of $Q$, and $\kappa$ is a bijection.

\subsubsection{The finite $\CAT$ cube complex $\E$}

The set $\{\mc{H}_i^{\pm}\mid 1\leq i\leq n\}$ naturally becomes a pocset when (partially) ordered by inclusion and equipped with the involution $\mc{H}_i^+\leftrightarrow\mc{H}_i^-$. Let $\E$ be the finite CAT(0) cube complex\footnote{This notation is chosen by analogy with the cube complex $\E_{\mathbf{\Pi}}$ appearing in \cite[Section~3.3]{CSV}.} dual to this pocset\footnote{Alternatively, we could simply consider the pocset that artificially declares the partitions $\mc{H}_i$ to be pairwise transverse. In other words, we could simply define $\E$ to be an entire $n$--cube (which results in some additional vertices, compared to our chosen definition). The entire construction in Subsections~\ref{subsec:host} and~\ref{subsec:extended_host} works in this case as well, without requiring any significant changes (except that the $\mc{H}_i$ should span an $n$--clique in $\wh\G$). As a result, one obtains a larger host $M'$ having the host $M$ constructed below as a (non-locally-convex) subcomplex and sharing all the same important properties. We prefer to define $\E$ as above since it seems more natural and yields a smaller host.}. We think of each edge of $\E$ as labelled by the vertex of $\wh\G$ corresponding to the $\mc{H}_i$ that it crosses. In addition, we orient each edge from the $\mc{H}_i^-$--side to the $\mc{H}_i^+$--side.

\subsubsection{The 1--skeleton of the host}

We now define an intermediate cube complex $M_1$ by adding edges to $\E$ as follows. Each edge $f$ added at this stage will be labelled by some colour $v\in\G$. In addition, the two germs of the edge $f$ at its endpoints will be labelled by the two oriented colours $v^-$ and $v^+$.

Consider a vertex $x\in\E^{(0)}$ and some colour $v\in\G$. Let $\eps_i\in\{\pm\}$ be the signs such that $\{x\}=\mc{H}_1^{\eps_1}\cap\dots\cap\mc{H}_n^{\eps_n}$, which exist by the construction of $\E$. 

If the generalised half-carrier $C(v^-)\sq Q$ is disjoint from some $\mc{H}_i^{\eps_i}\sq Q$, then there will be no edge germ labelled $v^+$ stemming from $x$.

If instead $C(v^-)\cap\mc{H}_i^{\eps_i}\neq\emptyset$ for all $i$, we define signs $\eps_1',\dots,\eps_n'$ so that $\eps_i'=-\eps_i$ exactly for those indices $i$ for which the hyperplanes of $\kappa^{-1}(v)$ appear within $\mc{H}_i^0$ (recall that each $\mc{H}_i^0$ contains either all or none of the hyperplanes in $\kappa^{-1}(v)$, due to the second sentence in Definition~\ref{defn:dividing_pattern}(1)). By Lemma~\ref{lem:WD1} immediately below, the subsets $\mc{H}_i^{\eps_i'}\sq Q$ pairwise intersect, so there exists a (unique) vertex $y\in\E^{(0)}$ such that $y\in\mc{H}_1^{\eps_1'}\cap\dots\cap\mc{H}_n^{\eps_n'}$. Now, we add to $\E$ an edge labelled by $v$ connecting $x$ to $y$, with the germ at $x$ labelled by $v^+$ and the germ at $y$ labelled by $v^-$.

We repeat this procedure at $x$ with the roles of $v^-$ and $v^+$ swapped. Then we repeat the entire procedure for every vertex of $\E$ and every colour $v\in\G$. 

We stress that we do not intend to create any multiple edges through these repetitions. Each element of $\G^{\pm}$ should label at most one edge germ at each vertex of $\E$. Remark~\ref{rmk:no_issues} below guarantees that there are no issues in this regard.

\begin{lem}\label{lem:WD1}
The subsets $\mc{H}_i^{\eps_i'}\sq Q$ pairwise intersect. Moreover, $C(v^+)\cap\mc{H}_i^{\eps_i'}\neq\emptyset$ for each $i$.
\end{lem}
\begin{proof}
If $\eps_i'=-\eps_i$, then $C(v^-)\sq\mc{H}_i^{\eps_i}$ and $C(v^+)\sq\mc{H}_i^{-\eps_i}$, in view of Definition~\ref{defn:dividing_pattern}(3). If instead $\eps_i'=\eps_i$, then there are three options: either $C(v)\sq\mc{H}_i^{\eps_i}$, or a hyperplane in $\kappa^{-1}(v)$ crosses one of the hyperplanes in $\mc{H}_i^0$, or finally $\kappa^{-1}(v)$ contains two hyperplanes $H,H'$ with $C(H)\sq\mc{H}_i^{\eps_i}$ and $C(H')\sq\mc{H}_i^{-\eps_i}$. In all three of these cases, $C(v^+)$ intersects $\mc{H}_i^{\eps_i}$.

From these observations, one immediately deduces that $C(v^+)\cap\mc{H}_i^{\eps_i'}\neq\emptyset$ for each $i$. Examining the three cases where $(\eps_i',\eps_j')\in\{(\eps_i,\eps_j),(-\eps_i,\eps_j),(-\eps_i,-\eps_j)\}$, one similarly sees that $\mc{H}_i^{\eps_i'}\cap\mc{H}_j^{\eps_j'}\neq\emptyset$ for all $i,j$. More precisely, if $(\eps_i',\eps_j')=(\eps_i,\eps_j)$, then $\mc{H}_i^{\eps_i'}\cap\mc{H}_j^{\eps_j'}=\mc{H}_i^{\eps_i}\cap\mc{H}_j^{\eps_j}$, which is nonempty since it contains the vertex $x$. In the remaining cases, we have $\eps_i'=-\eps_i$ and hence $C(v^+)\sq \mc{H}_i^{\eps_i'}$. As noted above, $C(v^+)$ intersects $\mc{H}_j^{\eps_j'}$ no matter what, so we again have $\mc{H}_i^{\eps_i'}\cap\mc{H}_j^{\eps_j'}\neq\emptyset$, completing the proof of the lemma.
\end{proof}

\begin{rmk}\label{rmk:no_issues}
The second part of Lemma~\ref{lem:WD1} implies that considering the pairs $(x,v^-)$ or $(y,v^+)$ actually results in the creation of the very same edge of $M_1$.
\end{rmk}

\subsubsection{Completing the host: higher-dimensional cubes}

Finally, starting from the cube complex $M_1$, which has $\wh\G$--labelled edges and $\wh\G^{\pm}$--labelled edge germs, we add squares and higher-dimensional cubes to ensure that it is non-positively curved, special, and that the crossing graph of its hyperplanes is exactly $\wh\G$. The result will be the host cube complex $M$.

To see that this is indeed possible, consider a vertex $x\in M_1^{(0)}=\E^{(0)}$ from which stem (among others) loops labelled by $a_1,\dots,a_k\in\G$, embedded edges labelled by $b_1,\dots,b_h\in\G$, and edges labelled by partitions $\mc{H}_{c_1},\dots,\mc{H}_{c_m}$. Suppose that all these labels span a clique in $\wh\G$. We would like to add a $(k+h+m)$--cube to $M_1$ spanned by these edges at $x$.

Let $\eps_i$ be the signs such that $\{x\}=\mc{H}_1^{\eps_1}\cap\dots\cap\mc{H}_n^{\eps_n}$. For each $1\leq j\leq h$, let $I_j\sq\{1,\dots,n\}$ be set of indices $i$ for which the hyperplanes in $\kappa^{-1}(b_j)$ appear in $\mc{H}_i^0$. Starting from $x$ and crossing the edge labelled $b_j$ has the effect of flipping exactly those signs $\eps_i$ with $i\in I_j$. Crossing an edge labelled by $\mc{H}_{c_j}$ flips only the sign $\eps_{c_j}$, while crossing one of the loops clearly does not affect signs at all. 

\begin{lem}\label{lem:WD2}
\begin{enumerate}
\setlength\itemsep{.25em}
\item[]
\item The sets $I_1,\dots,I_h,\{c_1\},\dots,\{c_m\}$ are pairwise disjoint. 
\item Starting from $x$, the signs indexed by the sets $I_1,\dots,I_h,\{c_1\},\dots,\{c_m\}$ can be flipped independently from each other, giving rise to $2^{h+m}$ distinct vertices of $M_1$ connected by an embedded copy of the $1$--skeleton of an $(h+m)$--cube with edges labelled by the $b_j$ and $\mc{H}_{c_j}$. 
\item At each of these $2^{h+m}$ vertices, there appear $k$ loops labelled by $a_1,\dots,a_k$. 
\end{enumerate}
\end{lem}
\begin{proof}
If we had $I_s\cap I_r\neq\emptyset$ for $s\neq r$, there would exist an index $i$ such that the hyperplanes in $\kappa^{-1}(b_s)\cup\kappa^{-1}(b_r)$ all appear in $\mc{H}_i^0$. This would contradict the fact that at least one hyperplane in $\kappa^{-1}(b_s)$ crosses a hyperplane in $\kappa^{-1}(b_r)$, which is the case because $b_s,b_r$ are joined by an edge of $\G$ by assumption and the special colouring is minimal. Similarly, if we had $c_r\in I_s$, the hyperplanes in $\kappa^{-1}(b_s)$ would all appear in $\mc{H}_{c_r}^0$, again contradicting the fact that $b_s$ and $\mc{H}_{c_r}$ are transverse (as defined at the end of Subsection~\ref{subsec:divisible}). This proves part~(1).

Now, fix some $s\in I_1$. Observe that the sets $\mc{H}_s^{\pm\eps_s}$ intersect all $\mc{H}_j^{\pm\eps_j}$ with $j\in I_2\cup\dots\cup I_h$. Indeed, $\mc{H}_s^0$ contains all hyperplanes in $\kappa^{-1}(b_1)$, the latter contains hyperplanes crossing hyperplanes in each of the sets $\kappa^{-1}(b_2),\dots,\kappa^{-1}(b_h)$, and one of these sets is entirely contained in $\mc{H}_j^0$. Thus, a hyperplane contained in $\mc{H}_s^0$ crosses a hyperplane contained in $\mc{H}_j^0$. Similarly, the sets $\mc{H}_s^{\pm\eps_s}$ intersect all $\mc{H}_{c_j}^{\pm\eps_{c_j}}$ with $1\leq j\leq m$, simply because $b_1$ is transverse to all $\mc{H}_{c_j}$, by assumption. This guarantees that, starting from $x$, the signs indexed by the sets $I_1,\dots,I_h,\{c_1\},\dots,\{c_m\}$ can be flipped independently from each other, giving rise to $2^{h+m}$ well-defined vertices of $M_1$.

In order to complete the proof of the lemma, we are left to show that, at each of these $2^{h+m}$ vertices, we see edges of $M_1$ labelled by $b_1,\dots,b_h$ and $a_1,\dots,a_k$. In fact, it suffices to check this at the vertices $y,z\in M_1^{(0)}=\E^{(0)}$ obtained from $x$ by flipping, respectively, only the signs indexed by $I_1$ or only the sign $\eps_{c_1}$ (then we can replace $x$ by $y$ or $z$ and repeat). Without loss of generality, let $b_1^+,\dots,b_h^+$ be the oriented colours labelling the relevant germs of edges at $x$.

Regarding $y$, we need to show that, for every $s\in I_1$, the subset $\mc{H}_s^{-\eps_s}\sq Q$ intersects the generalised half-carriers $C(b_2^-),\dots,C(b_h^-)$ and $C(a_1^-),\dots,C(a_k^-)$. This is clear, since $C(b_1^+)\sq\mc{H}_s^{-\eps_s}$ and $\kappa^{-1}(b_1)$ contains hyperplanes crossing hyperplanes in each of the sets $\kappa^{-1}(b_2),\dots,\kappa^{-1}(b_h)$ and $\kappa^{-1}(a_1),\dots,\kappa^{-1}(a_k)$.

Regarding $z$, we need to show that the set $\mc{H}_{c_1}^{-\eps_{c_1}}$ intersects the half-carriers $C(b_1^-),\dots,C(b_h^-)$ and $C(a_1^-),\dots,C(a_k^-)$. Again, this is clear since each of the colours $b_1,\dots,b_h$ and $a_1,\dots,a_k$ is transverse to the partition $\mc{H}_{c_1}$ (recall the discussion right before Definition~\ref{defn:extended_crossing_graph}). 

This completes the proof of the lemma.
\end{proof}

To the graph provided by Lemma~\ref{lem:WD2} we can attach the required $(k+h+m)$--cube. Repeating this construction for every vertex of $\E$ and every available clique of $\wh\G$ completes the construction of the host cube complex $M$ (avoiding duplicates, so that all cells of dimension $\geq 2$ are uniquely determined by their $1$--skeleton). The next proposition records various properties of the host $M$.

\begin{prop}\label{prop:host_properties}
Let $Q$ be a compact special cube complex with a fixed dividing pattern. Let $(\kappa,\G)$ be the involved special colouring and $\wh\G$ the extended crossing graph.
\begin{enumerate}
\setlength\itemsep{.25em}
\item The host cube complex $M$ is special and the crossing graph of its hyperplanes is $\wh\G$.
\item The finite $\CAT$ cube complex $\E$ embeds in $M$ as a locally convex subcomplex. Every edge of $M-\E$ crosses a hyperplane corresponding to a vertex of the subgraph $\G\sq\wh\G$.
\item Collapsing the hyperplanes of $M$ labelled by $\wh\G-\G$ (i.e.\ those crossing $\E$) gives a homotopy equivalence from $M$ to the Salvetti complex of the RAAG $A_{\G}$. In particular, $\pi_1(M)\cong A_{\G}$.
\item There is a natural embedding $Q\hookrightarrow M$ as a locally convex subcomplex. Thus $\pi_1(Q)\hookrightarrow\pi_1(M)$.
\item If an edge labelled by a vertex of $\G\sq\wh\G$ stems from a vertex of $Q\sq M$, it is contained in $Q$.
\end{enumerate}
\end{prop}
\begin{proof}
The fact that $M$ is special is immediate from its construction. Indeed, each of its edge germs is naturally labelled by an element of $\wh\G^{\pm}$, no two edge germs at a given vertex have the same label, and labels completely regulate whether two edge germs span a square germ or not.

In order to show that $\wh\G$ is the crossing graph of the hyperplanes of $M$, it suffices to show that all edges of $M$ labelled by a given colour $v\in\G\sq\wh\G$ are actually dual to the same hyperplane of $M$.

Let us prove a slightly stronger statement, whose full strength will be required later in the proof.

\smallskip
{\bf Claim~1.} \emph{Consider two vertices $x,x'\in M$ from which stem edges $f,f'$ with germs labelled by some $v^+\in\G^{\pm}$. Then, there exist edges $f_0=f,f_1,\dots,f_k=f'\sq M$ such that $f_i$ and $f_{i+1}$ are opposite in a square of $M$ whose other pair of opposite edges is contained in the subcomplex $\E\sq M$.}

\smallskip
\emph{Proof of Claim~1.}
Say that $\{x\}=\mc{H}_1^{\eps_1}\cap\dots\cap\mc{H}_n^{\eps_n}$ and $\{x'\}=\mc{H}_1^{\eps_1'}\cap\dots\cap\mc{H}_n^{\eps_n'}$ for signs $\eps_i,\eps_i'\in\{\pm\}$. Let $x_0=x,x_1,\dots,x_k=x'$ be the vertices of a geodesic in $\E$ from $x$ to $x'$. Recall that $\E$ is $\CAT$ and the $\mc{H}_i$ are its hyperplanes. Thus, if we have $\eps_i=\eps_i'$ for some $i$, then the vertices $x_1,\dots,x_{k-1}$ all lie in $\mc{H}_i^{\eps_i}$. It follows that the half-carrier $C(v^-)\sq Q$ intersects each set $\mc{H}_i^{\pm}\sq Q$ containing one of the $x_j\in\E$. Hence there is an edge $f_i$ with germ labelled $v^+$ stemming from each $x_i$.

If $\eps_i'=-\eps_i$, then $C(v^-)$ intersects both $\mc{H}_i^{\eps_i}$ and $\mc{H}_i^{-\eps_i}$, which implies that the colour $v$ must be transverse to $\mc{H}_i$. In other words, $v$ and $\mc{H}_i$ are connected by an edge of $\wh\G$, which implies that edges $f_j$ with consecutive indices span a square of $M$ whose remaining edges lie in $\E$, as required.
\hfill$\blacksquare$

\smallskip
This completes the proof of part~(1). Part~(2) is also immediate from the construction of $M$.

Let us prove part~(3). The hyperplanes labelled by $\wh\G-\G$ (i.e.\ those that cross edges of $\E$) are all carrier retracts. Moreover, any loop crossing only such hyperplanes is entirely contained in the $\CAT$ subcomplex $\E$ and thus nulhomotopic. It follows that we can collapse all these hyperplanes and obtain a homotopy equivalence $M\twoheadrightarrow S$, where $S$ is a non-positively curved cube complex whose hyperplane set is naturally in bijection with $\G^{(0)}$ (see e.g.\ Lemmas~4.4 and~4.5 in \cite{CSV} for this). Since $M^{(0)}=\E^{(0)}$, the cube complex $S$ has a single vertex. 

In order to conclude that $S$ is the required Salvetti complex, we need to check that it has only one edge dual to each hyperplane, and only one $k$--cube for each $k$--clique in $\G$. The former follows from Claim~1. The latter is a consequence of the following claim, of which we do not give a proof to avoid cumbersome notation; the argument is entirely analogous to the one used for Claim~1. We again state a more general claim than is strictly necessary, since this will be needed in Lemma~\ref{lem:extended_host_collapses} below.

\smallskip
{\bf Claim~2.} \emph{Consider two $m$--cells $c,c'\sq M$ based at vertices $x,x'\in M$ and spanned by edge germs with the same labels (possibly coming from both $\G^{\pm}$ and $\wh\G^{\pm}-\G^{\pm}$). Then there exist a geodesic in $\E$ with vertices $x_0=x,x_1,\dots,x_k=x'$ and $m$--cells $c_i$ based at $x_i$, all with the same edge germ labels, and such that consecutive $c_i$'s lie in an $(m+1)$--cell whose additional hyperplane is a vertex of $\wh\G-\G$.}

\smallskip
Now, we prove part~(4). First, note that there is a natural embedding $j\colon Q^{(0)}\hookrightarrow M^{(0)}$ defined as follows. For every vertex $x\in Q^{(0)}$ and every $1\leq i\leq n$, there is a unique sign $\eps_i(x)\in\{\pm\}$ such that $x\in\mc{H}_i^{\eps_i(x)}$. The subsets $\mc{H}_1^{\eps_1(x)},\dots,\mc{H}_n^{\eps_n(x)}\sq Q$ pairwise intersect (as they contain $x$), so they define a (unique) vertex of $\E^{(0)}$. We define $j(x)$ as this vertex. The resulting map $j$ is injective since distinct vertices of $Q$ are separated by at least one of the partitions $\mc{H}_i$, by definition of dividing pattern.

If $x,y\in Q^{(0)}$ are joined by an edge dual to a hyperplane $H$ with $\kappa(H)=v$, it is clear from the construction of $M$ that $j(x)$ and $j(y)$ are connected by a (unique) edge of $M$ labelled by $v$. Thus, we obtain a canonical embedding $Q^{(1)}\hookrightarrow M^{(1)}$ (for injectivity, recall that the oriented hyperplanes in each $\kappa^{-1}(v^{\pm})$ have pairwise-disjoint positive carriers in $Q$). Since $M$ is special, $\G\sq\wh\G$ is a full subgraph and $(\kappa,\G)$ is a special colouring for $Q$, it follows that this embedding of $1$--skeletons uniquely extends to an embedding $Q\hookrightarrow M$ with locally convex image. 

Locally convex subcomplexes of non-positively curved cube complexes are $\pi_1$--injective, which also shows that $\pi_1(Q)\hookrightarrow\pi_1(M)$, completing the proof of part~(4).

Finally, we prove part~(5). Consider an edge $f\sq M$ stemming from a vertex $x\in Q\sq M$ and suppose that the germ at $x$ is labelled by some oriented colour $v^+$ (without loss of generality). This means that $C(v^-)\cap\mc{H}_i^{\eps_i}\neq\emptyset$ for all $i$, where $\{x\}=\mc{H}_1^{\eps_1}\cap\dots\cap\mc{H}_n^{\eps_n}$. In view of the second half of part~(4) of Definition~\ref{defn:dividing_pattern}, it follows that, within $Q$, we must have $x\in C(v^-)$. Thus, $Q$ already contains an edge with germ at $x$ labelled $v^+$, and this edge must then coincide with $f$. In conclusion, we have $f\sq Q$ as required.
\end{proof}

\begin{rmk}\label{rmk:why_important}
Up until part~(5) of Proposition~\ref{prop:host_properties}, we did not use much of the definition of ``divisible'' cube complexes. Everything would have worked for any compact special cube complex $Q$ satisfying part~(1) of Definition~\ref{defn:strongly_divisible}, selecting a few collections of pairwise-disjoint hyperplanes $\mc{H}_i$ separating any two vertices of $Q$ and running the above construction.

The importance of part~(5) of Proposition~\ref{prop:host_properties} (and hence of the rest of the Definitions~\ref{defn:strongly_divisible} and~\ref{defn:dividing_pattern}) will become clear in the next subsection.
\end{rmk}

\subsection{The extended host and its automorphisms}\label{subsec:extended_host}

In the previous subsection, we have embedded the divisible cube complex $Q$ into its host $M$. Now, we will duplicate some of the edges of $M$ to form an ``extended host'' $\wh M$. This will have similar properties to $M$ and will additionally admit a cubical automorphism whose fixed subset has $Q$ as one of its connected components. 

Fix an integer $N\geq 2$. Define a cube complex $\wh M_N$ as follows. Informally, $\wh M_N$ is obtained from $M$ by replacing every edge of $\E\sq M$ with a $\Theta$--graph with $N$ strands.

More precisely, $M$ and $\wh M_N$ have the same $0$--skeleton and exactly the same edges labelled by vertices of $\G\sq\wh\G$. Every edge $f\sq M$ labelled by a vertex of $\wh\G-\G$ is replaced by $N$ distinct copies of itself, denoted $f^1,\dots,f^N\sq\wh M_N$. If $M$ contains a cube $c$ spanned by $\G$--labelled edges $u_1,\dots,u_m$ and $(\wh\G-\G)$--labelled edges $f_1,\dots,f_k$ stemming from a vertex $x$, then $\wh M_N$ has in its place $N^k$ cubes of the same dimension, each spanned by the edges $u_1,\dots,u_m,f_1^{i_1},\dots,f_k^{i_k}$ for one of the possible choices of $1\leq i_1,\dots,i_k\leq N$. Moreover, this construction is performed ensuring that, if $g_1$ and $g_2$ are opposite $(\wh\G-\G)$--labelled edges in a square of $M$, then $g_1^i$ and $g_2^i$ are opposite edges in squares of $\wh M_N$ for each $1\leq i\leq N$, and we never have some $g_1^i$ opposite to some $g_2^j$ with $i\neq j$.

It is straightforward to see that $\wh M_N$ is non-positively curved and special. The crossing graph of its hyperplanes is the graph $\wh\G_N$ having $\G$ as a full subgraph and additional vertices $\mc{H}_i^s$ for $1\leq i\leq n$ and $1\leq s\leq N$, with the following edges:
\begin{itemize}
\item $[v,\mc{H}_i^s]$ for $v\in\G$ if and only if $[v,\mc{H}_i]\sq\wh\G$;
\item $[\mc{H}_i^s,\mc{H}_{i'}^{s'}]$ if and only if $i\neq i'$ and $[\mc{H}_i,\mc{H}_{i'}]\sq\wh\G$.
\end{itemize}

Much like $M$, the complex $\wh M_N$ naturally collapses to a Salvetti complex.

\begin{lem}\label{lem:extended_host_collapses}
Collapsing the hyperplanes of $\wh M_N$ corresponding to $\mc{H}_1^1,\dots,\mc{H}_n^1\in\wh\G_N$ gives a homotopy equivalence from $\wh M_N$ to the Salvetti complex of the RAAG $A_{\wh\G_{N-1}}$. Thus, $\pi_1(\wh M_N)\cong A_{\wh\G_{N-1}}$.
\end{lem}
\begin{proof}
The proof is exactly the same as that of part~(3) of Proposition~\ref{prop:host_properties}, using Claim~2 instead of Claim~1. Namely, if $c,c'\sq M$ are two $m$--cells based at vertices $x,x'\in M$ and spanned by edge germs with the same labels, then there is a sequence of cells $c_i\sq M$ starting with $c$, ending with $c'$, and with consecutive cells spanning an $(m+1)$--cell whose additional hyperplane is a vertex of $\wh\G-\G$. These paths of cells in $M$ yield analogous paths of cells in $\wh M_N$, where we can take the hyperplanes separating consecutive $c_i$'s to be of the form $\mc{H}^1_j$, so that they get collapsed. Ultimately, one obtains that the collapse of $\wh M_N$ has a unique cell for each clique in $\wh\G_N-\{\mc{H}_1^1,\dots,\mc{H}_n^1\}$. The full subgraph of $\wh\G_N$ avoiding $\{\mc{H}_1^1,\dots,\mc{H}_n^1\}$ is isomorphic to $\wh\G_{N-1}$, so this proves the lemma.
\end{proof}

Now, there is a natural order--$N$ cubical automorphism $\Phi\colon\wh M_N\ra\wh M_N$ that fixes pointwise the $0$--skeleton and every $\G$--labelled edge, while it cyclically permutes the $N$ edges originating from any given $(\wh\G-\G)$--labelled edge of $M$. To fix the definition of $\Phi$, let us say that $\Phi(f^i)=f^{i+1}$ for each such edge $f\sq M$, with indices considered modulo $N$. By part~(5) of Proposition~\ref{prop:host_properties}, the subcomplex $Q\sq \wh M_N$ is an entire connected component of $\Fix(\Phi)\sq\wh M_N$ (cf.\ Remark~\ref{rmk:why_important}).

From this, we immediately deduce the following result, proving one half of Theorem~\ref{thmintro:divisible_vs_fix}.

\begin{thm}\label{thm:divisible->fix}
Let $G$ be the fundamental group of a divisible compact special cube complex. Let $\wh\G$ be the extended crossing graph of a dividing pattern. Fix some $N\geq 2$ and set $\Lambda:=\wh\G_{N-1}$. Then the RAAG $A_{\Lambda}$ admits an order--$N$ untwisted automorphism $\varphi\in U(A_{\Lambda})\leq\Aut(A_{\Lambda})$ such that $\Fix(\varphi)\cong G$. For $N=2$, we have $\varphi\in U^0(A_{\wh\G})$; in particular, $\varphi$ is pure.
\end{thm}
\begin{proof}
Realise $G$ as the fundamental group of a divisible compact special cube complex $Q$. Choose a dividing pattern and apply the above construction to form the extended host $\wh M_N$ with its order--$N$ cubical automorphism $\Phi$. Since $Q\sq\wh M_N$ is a connected component of the fixed set of $\Phi$, we can apply Lemma~\ref{lem:pi1_vs_fix}(2) and obtain an order--$N$ automorphism $\varphi$ of $\pi_1(\wh M_N)\cong A_{\wh\G_{N-1}}$ with $\Fix(\varphi)\cong \pi_1(Q)$. Since $\Phi$ is a cubical automorphism of $\wh M_N$, the induced automorphism $\varphi$ is coarse-median preserving, hence untwisted by \cite{Fio10a}. 

For $N=2$, it is easily seen that $\varphi$ is a product of Whitehead automorphisms of $A_{\Lambda}$ (see Section~2.3 and Lemma~3.2 in \cite{CSV}), so it is pure. We also have $\Lambda=\wh\G_1=\wh\G$, which concludes the proof.
\end{proof}

\begin{ex}\label{ex:explicit}
The proof of Theorem~\ref{thm:divisible->fix} allows us to give a simple and explicit description of automorphisms of RAAGs with interesting fixed subgroups. We consider $N=2$ throughout.
\begin{enumerate}
\setlength\itemsep{.25em}
\item Let $S_2$ be the closed orientable surface of genus $2$. As discussed in Example~\ref{ex:strongly_divisible}, the cube complex $Q_2$ dual to the curves in Figure~\ref{fig:surface_2} is special and divisible.

We can construct a dividing pattern where $(\kappa,\G)$ is a standard special colouring and there are just two partitions $\mc{H}_1,\mc{H}_2$. Namely, $\mc{H}_1^0$ and $\mc{H}_2^0$ are the union of the blue and red curves, respectively. 

The extended crossing graph $\wh\G$ of this dividing pattern is depicted in Figure~\ref{fig:surface_2} on the right. Each vertex of $\wh\G$ has been drawn in the colour of the corresponding curves, and the vertices corresponding to $\mc{H}_1,\mc{H}_2$ are the two thicker ones $B,R$. Note that the standard crossing graph of the hyperplanes of $Q_2$, i.e.\ the subgraph $\G\sq\wh\G$, is simply a hexagon with vertices alternately coloured blue and red.

Now, Theorem~\ref{thm:divisible->fix} shows that there exists an order--$2$ automorphism $\varphi\in U^0(A_{\wh\G})$ with $\Fix(\varphi)\cong\pi_1(S_2)$. Explicitly, $\varphi$ is determined by the following assignments (see Section~2.3 and Lemma~3.2 in \cite{CSV}):
\begin{align*}
\varphi(B)&=B^{-1}, & \varphi(b_i)&=Bb_i, & \varphi(R)&=R^{-1}, & \varphi(r_i)&=Rr_i.
\end{align*}
Moreover, we have $\Fix(\varphi)=\langle b_3^{-1}b_1,\ b_3^{-1}b_2,\ r_3^{-1}r_1,\ r_3^{-1}r_2 \rangle$.

\begin{figure}[ht]
\begin{tikzpicture}[baseline=(current bounding box.center)]
\draw[thick] (0,0) ellipse (4cm and 2.5cm);
\draw[thick] (-1.75,0) ellipse (1cm and 1cm);
\draw[thick] (1.75,0) ellipse (1cm and 1cm);
\tikzset{
    partial ellipse/.style args={#1:#2:#3}{
        insert path={+ (#1:#3) arc (#1:#2:#3)}
    }
}
\draw[ultra thick, blue] (0,0) [partial ellipse=0:180:0.75cm and 0.4cm];
\draw[ultra thick, blue, dashed] (0,0) [partial ellipse=180:360:0.75cm and 0.4cm];
\draw[ultra thick, blue] (-3.375,0) [partial ellipse=0:180:0.625cm and 0.35cm];
\draw[ultra thick, blue, dashed] (-3.375,0) [partial ellipse=180:360:0.625cm and 0.35cm];
\draw[ultra thick, blue] (3.375,0) [partial ellipse=0:180:0.625cm and 0.35cm];
\draw[ultra thick, blue, dashed] (3.375,0) [partial ellipse=180:360:0.625cm and 0.35cm];
\draw[very thick, red] (0,0) ellipse (3.75cm and 2.25cm);
\draw[very thick, red] (-1.75,0) ellipse (1.25cm and 1.25cm);
\draw[very thick, red] (1.75,0) ellipse (1.25cm and 1.25cm);
\end{tikzpicture}
\hspace{2cm}
\begin{tikzpicture}[baseline=(current bounding box.center)]
\draw[fill,blue] (0,0) circle [radius=0.12cm];
\node [below left] at (0,0) {\blue{$B$}};
\draw[fill,red] (0,2) circle [radius=0.08cm];
\node [above left] at (0,2) {\red{$r_3$}};
\draw[fill,red] (2,0) circle [radius=0.08cm];
\node [above left] at (2,0) {\red{$r_1$}};
\draw[fill,blue] (2,2) circle [radius=0.08cm];
\node [above] at (2,2) {\blue{$b_2$}};
\draw[fill,blue] (1.2,2.9) circle [radius=0.08cm];
\node [above] at (1.2,2.9) {\blue{$b_1$}};
\draw[fill,red] (3.2,2.9) circle [radius=0.12cm];
\node [above right] at (3.2,2.9) {\red{$R$}};
\draw[fill,red] (1.2,0.9) circle [radius=0.08cm];
\node [above left] at (1.2,0.9) {\red{$r_2$}};
\draw[fill,blue] (3.2,0.9) circle [radius=0.08cm];
\node [above left] at (3.2,0.9) {\blue{$b_3$}};
\draw[thick] (0,0) -- (0,2);
\draw[thick] (0,2) -- (2,2);
\draw[thick] (2,2) -- (2,0);
\draw[thick] (2,0) -- (0,0);
\draw[thick] (0,2) -- (1.2,2.9);
\draw[thick] (1.2,2.9) -- (3.2,2.9);
\draw[thick] (3.2,2.9) -- (2,2);
\draw[thick] (3.2,2.9) -- (3.2,0.9);
\draw[thick] (3.2,0.9) -- (2,0);
\draw[thick,dashed] (0,0) -- (1.2,0.9);
\draw[thick,dashed] (1.2,0.9) -- (1.2,2.9);
\draw[thick,dashed] (1.2,0.9) -- (3.2,0.9);
\tikzset{
    partial ellipse/.style args={#1:#2:#3}{
        insert path={+ (#1:#3) arc (#1:#2:#3)}
    }
}
\draw[thick,rotate around={43:(1.6,1.45)}] (1.6,1.45) [partial ellipse=180:360:2.16cm and 2.1cm];
\end{tikzpicture}
\caption{On the left: curves on $S_2$ whose dual square complex $Q_2$ is divisible. On the right: the extended crossing graph $\wh\G$ of the simplest dividing pattern for $Q_2$.}
\label{fig:surface_2}
\end{figure}
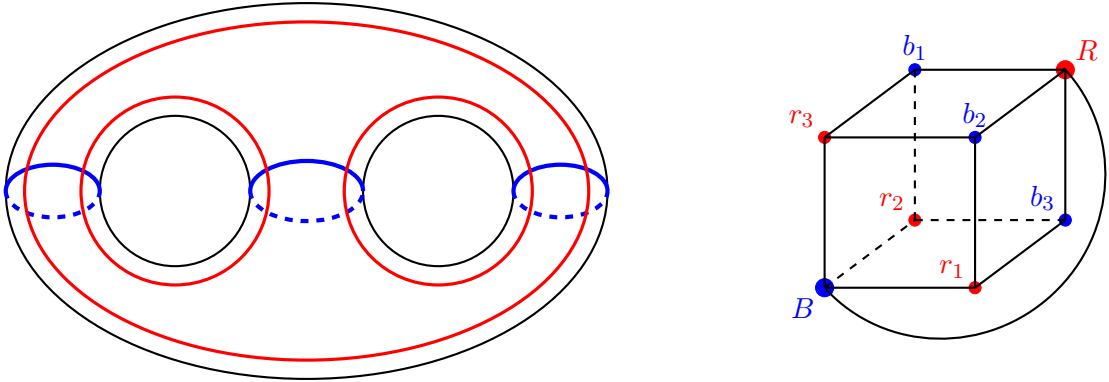

\item Consider the right-angled Coxeter group $W_{\G}$, its commutator subgroup $W_{\G}'$, and the divisible cube complex $Q_{\G}$ with $\pi_1(Q_{\G})\cong W_{\G}'$ described in Example~\ref{ex:strongly_divisible}.

A simple dividing pattern for $Q_{\G}$ can be constructed as follows. First, we consider the natural special colouring $\kappa\colon\mscr{W}(Q_{\G})\ra\G$ where each preimage $\kappa^{-1}(v)$ is precisely the set of hyperplanes of $Q_{\G}$ dual to edges of $Q_{\G}$ labelled by $v$. As in Example~\ref{ex:strongly_divisible}, we also let $\mc{H}_v^0$ be the union of the hyperplanes in $\kappa^{-1}(v)$, and denote by $\mc{H}_v^-,\mc{H}_v^+$ the connected components of $Q_{\G}-\mc{H}_v^0$ containing $\underline{0}$ and $\underline{1}$, respectively.

The extended crossing graph $\wh\G$ for this dividing pattern has the following structure. The vertex set of $\wh\G$ consists of two copies of the vertex set of $\G$, one corresponding to the colours of the special colouring and one corresponding to the $\mc{H}_v$. We therefore denote these two copies of $\G^{(0)}$ by $\G_{\kappa}^{(0)}\sq\wh\G^{(0)}$ and $\G_{\mc{H}}^{(0)}\sq\wh\G^{(0)}$, respectively. For a vertex $v\in\G$, we also write $v_{\kappa},v_{\mc{H}}$ for its two copies in $\wh\G$. The full subgraph $\G_{\kappa}\sq\wh\G$ spanned by the vertices in $\G_{\kappa}^{(0)}$ is simply isomorphic to $\G$. Instead, the vertices in $\G_{\mc{H}}^{(0)}$ span an entire clique $\G_{\mc{H}}\sq\wh\G$. Finally, there are additional edges $[v_{\kappa},w_{\mc{H}}]$ whenever $v\neq w$. An example of the construction of $\wh\G$ from $\G$ is depicted in Figure~\ref{fig:RACG}.

By Theorem~\ref{thm:divisible->fix}, there exists an order--$2$ automorphism $\varphi\in U^0(A_{\wh\G})$ with $\Fix(\varphi)\cong W_{\G}'$. Explicitly, we have
\begin{align*}
\varphi(v_{\kappa})&=v_{\mc{H}}v_{\kappa}, & \varphi(v_\mc{H})&=v_\mc{H}^{-1},
\end{align*}
for each $v\in\G$.

\begin{figure}[ht]
\begin{tikzpicture}[baseline=(current bounding box.center)]
\draw[fill] (2*1,2*0) circle [radius=0.07cm];
\node [right] at (2*1,2*0) {$a$};
\draw[fill] ({2*cos(72)},{2*sin(72)}) circle [radius=0.07cm];
\node [above] at ({2*cos(72)},{2*sin(72)}) {$b$};
\draw[fill] ({2*cos(144)},{2*sin(144)}) circle [radius=0.07cm];
\node [above] at ({2*cos(144)},{2*sin(144)}) {$c$};
\draw[fill] ({2*cos(216)},{2*sin(216)}) circle [radius=0.07cm];
\node [left] at ({2*cos(216)},{2*sin(216)}) {$d$};
\draw[fill] ({2*cos(288)},{2*sin(288)}) circle [radius=0.07cm];
\node [below] at ({2*cos(288)},{2*sin(288)}) {$e$};
\draw[fill,thick] (2*1,2*0) -- ({2*cos(72)},{2*sin(72)});
\draw[fill,thick] ({2*cos(72)},{2*sin(72)}) -- ({2*cos(144)},{2*sin(144)});
\draw[fill,thick] ({2*cos(144)},{2*sin(144)}) -- ({2*cos(216)},{2*sin(216)});
\draw[fill,thick] ({2*cos(216)},{2*sin(216)}) -- ({2*cos(288)},{2*sin(288)});
\draw[fill,thick] ({2*cos(288)},{2*sin(288)}) -- (2*1,2*0);
\end{tikzpicture}
\hspace{2.5cm}
\begin{tikzpicture}[baseline=(current bounding box.center)]
\draw[fill] (2*1,2*0) circle [radius=0.07cm];
\node [right] at (2*1,2*0) {$a_{\kappa}$};
\draw[fill] ({2*cos(72)},{2*sin(72)}) circle [radius=0.07cm];
\node [above] at ({2*cos(72)},{2*sin(72)}) {$b_{\kappa}$};
\draw[fill] ({2*cos(144)},{2*sin(144)}) circle [radius=0.07cm];
\node [above] at ({2*cos(144)},{2*sin(144)}) {$c_{\kappa}$};
\draw[fill] ({2*cos(216)},{2*sin(216)}) circle [radius=0.07cm];
\node [left] at ({2*cos(216)},{2*sin(216)}) {$d_{\kappa}$};
\draw[fill] ({2*cos(288)},{2*sin(288)}) circle [radius=0.07cm];
\node [below] at ({2*cos(288)},{2*sin(288)}) {$e_{\kappa}$};
\draw[fill,thick] ({2*cos(0)},{2*sin(0)}) -- ({2*cos(72)},{2*sin(72)});
\draw[fill,thick] ({2*cos(72)},{2*sin(72)}) -- ({2*cos(144)},{2*sin(144)});
\draw[fill,thick] ({2*cos(144)},{2*sin(144)}) -- ({2*cos(216)},{2*sin(216)});
\draw[fill,thick] ({2*cos(216)},{2*sin(216)}) -- ({2*cos(288)},{2*sin(288)});
\draw[fill,thick] ({2*cos(288)},{2*sin(288)}) -- ({2*cos(0)},{2*sin(0)});
\foreach \i in {0,1,2,3,4}{ 
	\foreach \j in {0,1,2,3,4}{
	\draw[fill,thick] ({0.6*cos(180+\i*72)},{0.6*sin(180+\i*72)}) -- ({0.6*cos(180+\j*72)},{0.6*sin(180+\j*72)}); }}
\foreach \i in {0,1,2,3,4}{ 
	\foreach \j in {2,3}{
	\draw[fill] ({0.6*cos(180+\i*72)},{0.6*sin(180+\i*72)}) -- ({2*cos((\i+\j)*72)},{2*sin((\i+\j)*72)}); }}
\foreach \i in {0,1,2,3,4}{ 
	\foreach \j in {4}{
	\draw plot [smooth] coordinates {({0.6*cos(180+\i*72)},{0.6*sin(180+\i*72)})     ({   0.3*cos(180+\i*72)+cos((\i+\j)*72)  +0.1*(2*sin((\i+\j)*72)-0.6*sin(180+\i*72))    } , {   0.3*sin(180+\i*72)+sin((\i+\j)*72)  -0.1*(2*cos((\i+\j)*72)-0.6*cos(180+\i*72))   })    ({2*cos((\i+\j)*72)},{2*sin((\i+\j)*72)})}; }}
\foreach \i in {0,1,2,3,4}{ 
	\foreach \j in {1}{
	\draw plot [smooth] coordinates {({0.6*cos(180+\i*72)},{0.6*sin(180+\i*72)})     ({   0.3*cos(180+\i*72)+cos((\i+\j)*72)  -0.1*(2*sin((\i+\j)*72)-0.6*sin(180+\i*72))    } , {   0.3*sin(180+\i*72)+sin((\i+\j)*72)  +0.1*(2*cos((\i+\j)*72)-0.6*cos(180+\i*72))   })    ({2*cos((\i+\j)*72)},{2*sin((\i+\j)*72)})}; }}
\draw[fill] ({0.6*cos(180+0)},{0.6*sin(180+0)}) circle [radius=0.07cm];
\node [left] at (({0.6*cos(180+0)},{0.6*sin(180+0)}) {$a_{\mc{H}}$};
\draw[fill] ({0.6*cos(180+72)},{0.6*sin(180+72)}) circle [radius=0.07cm];
\node [below] at ({0.6*cos(180+72)},{0.6*sin(180+72)}) {$b_{\mc{H}}$};
\draw[fill] ({0.6*cos(180+144)},{0.6*sin(180+144)}) circle [radius=0.07cm];
\node [below] at ({0.6*cos(180+144)},{0.6*sin(180+144)}) {$c_{\mc{H}}$};
\draw[fill] ({0.6*cos(180+216)},{0.6*sin(180+216)}) circle [radius=0.07cm];
\node [right] at ({0.6*cos(180+216)},{0.6*sin(180+216)}) {$d_{\mc{H}}$};
\draw[fill] ({0.6*cos(180+288)},{0.6*sin(180+288)}) circle [radius=0.07cm];
\node [above] at ({0.6*cos(180+288)},{0.6*sin(180+288)}) {$e_{\mc{H}}$};
\end{tikzpicture}
\caption{On the left: a graph $\G$. On the right: the graph $\wh\G$ constructed in Ex\-ample~\ref{ex:explicit}(2) such that there exists an involution $\varphi\in U^0(A_{\wh\G})$ with $\Fix(\varphi)\cong W_{\G}'$.}
\label{fig:RACG}
\end{figure}
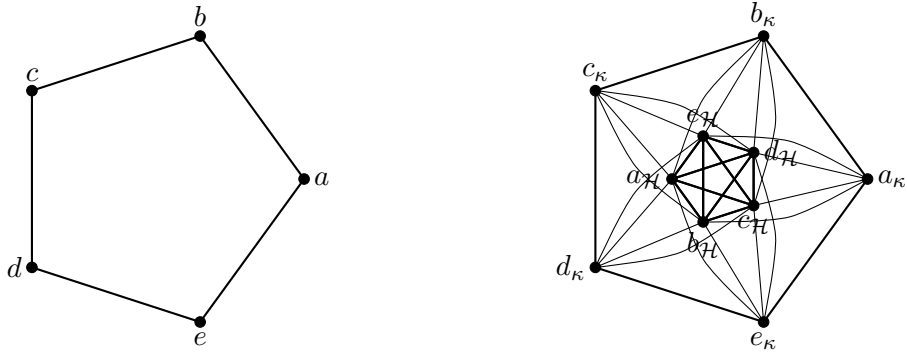

\item Consider the graph braid group $B_n(\G)$, where $2\leq n\leq\#\G^{(0)}-2$.

The extended crossing graph $\wh\G$ of the dividing pattern for $UC_n(\G)$ constructed in Example~\ref{ex:braid_divisible} can be described as follows. The graph $\wh\G$ has a vertex $\overline e$ for every edge $e\sq\G$ (cor\-responding to the colours of the special colouring) and a vertex $\overline a$ for every vertex $a\in\G$ (cor\-responding to the partitions in the dividing pattern), as well as the following edges:
\begin{itemize}
\item $[\overline e,\overline f]$ for any two edges $e,f\sq\G$ with $e\cap f=\emptyset$;
\item $[\overline a,\overline b]$ for any two distinct vertices $a,b\in\G$;
\item $[\overline e,\overline a]$ for any edge $e\sq\G$ and any vertex $a\in\G$ such that $a\not\in e$.
\end{itemize}

Note that the graph $\wh\G$ --- as well as the host $M$, the extended host $\wh M_2$ and the cubical automorphism $\Phi_2\in\Aut(\wh M_2)$ --- are all completely independent of the integer $n$ defining the graph braid group $B_n(\G)$. In fact, it is easy to see that $\Fix(\Phi_2)\sq\wh M_2$ has exactly $\#\G^{(0)}+1$ connected components, respectively isomorphic to the configuration spaces $UC_n(\G)$ for the various integers of $0\leq n\leq\#\G^{(0)}$. We now explain this point in more detail.

First, recall that the dividing pattern for $UC_n(\G)$ consists of a partition $\mc{H}_a$ for each vertex $a\in\G$; the vertices of $UC_n(\G)$ in the subset $\mc{H}_a^+$ are represented by cardinality--$n$ subsets of $\G^{(0)}$ containing $a$, while vertices in $\mc{H}_a^-$ are cardinality--$n$ subsets of $\G^{(0)}$ \emph{not} containing $a$. Since $2\leq n\leq\#\G^{(0)}-2$, any two distinct partitions are transverse. Thus, the $\CAT$ cube complex $\E$ is a cube of dimension $\#\G^{(0)}$, irrespective of the value of $n$. Each vertex of $\E$ corresponds to a choice, for each vertex $a\in\G$, of either the side $\mc{H}_a^-$ or the side $\mc{H}_a^+$; thus, the set $\E^{(0)}=M^{(0)}$ is naturally identified with the entire power set of $\G^{(0)}$. Thinking of vertices of $M$ as subsets of $\G^{(0)}$ (of any cardinality), each edge of $M-\E$ replaces a given vertex of $\G$ with a different one, and thus it connects subsets of $\G^{(0)}$ with the same cardinality. Hence, removing the edges of $\E$, the host $M$ gets split into $\#\G^{(0)}+1$ connected components, depending on the cardinality of the subsets of $\G^{(0)}$ represented by the vertices of $\E$; these are also the components of $\Fix(\Phi_2)$ within the extended host $\wh M_2$. Each of these components is clearly isomorphic to $UC_n(\G)$ for the respective value of $n$.

In conclusion, along with Lemma~\ref{lem:pi1_vs_fix} and Theorem~\ref{thm:divisible->fix}, all this shows that the RAAG $A_{\wh\G}$ has involutions $\varphi_n\in U^0(A_{\wh\G})$ for $0\leq n\leq\#\G^{(0)}$, each with $\Fix(\varphi_n)\cong B_n(\G)$. Moreover, these involutions all determine the same outer automorphism of $A_{\wh\G}$. Indeed, the difference between the $\varphi_n$ is only due to the choice of a basepoint for the fundamental group of $\wh M_2$, which can be picked in any of the $\#\G^{(0)}+1$ distinct connected components of the subset $\Fix(\Phi_2)\sq\wh M_2$. 

Unlike the previous two examples, we will not give explicit formulas for these involutions. This would require picking a basepoint of $\wh M_2$ in a specific component of $\Fix(\Phi_2)$, which cannot be done canonically in general. As a consequence, the involutions $\varphi_n$ do not appear to have particularly canonical or pleasant expressions.
\end{enumerate}
\end{ex}

\subsection{The converse}\label{subsec:converse}

In this subsection, we complete the proof of Theorem~\ref{thmintro:divisible_vs_fix} by obtaining the following converse to Theorem~\ref{thm:divisible->fix}. We will use in an essential way the fact that finite subgroups of $U^0(A_{\G})$ fix points of untwisted Outer Space, which was recently proved by Bregman, Charney and Vogtmann \cite{BCV2} partly building on ideas of Hensel and Kielak \cite{HK1,HK2}.

\begin{prop}\label{prop:fix->divisible}
Let $A_{\G}$ be a right-angled Artin group. Consider:
\begin{itemize}
\item either a pure untwisted automorphism $\varphi\in U^0(A_{\G})$ with finite-order projection $[\varphi]\in\Out(A_{\G})$;
\item or a finite group of pure untwisted automorphisms $F<U^0(A_{\G})\leq\Aut(A_{\G})$. 
\end{itemize}
Then $\Fix(\varphi)$ and $\Fix(F)$ are fundamental groups of divisible compact special cube complexes.
\end{prop}

Before proving the proposition, we need to obtain a couple of lemmas.

\begin{lem}\label{lem:finite-order_reduction}
Let $\varphi\in U(A_{\G})\leq \Aut(A_{\G})$ be an untwisted automorphism such that the projection $[\varphi]\in\Out(A_{\G})$ has finite order. Then there exists a $\varphi$--invariant parabolic subgroup $P\leq A_{\G}$ such that $\varphi|_P\in\Aut(P)$ has finite order and $\Fix(\varphi)\cong\Fix(\varphi|_P)\x\Z^k$ for some $k\geq 0$. Moreover, the automorphism $\varphi|_P$ is untwisted and, in case $\varphi$ was pure, $\varphi|_P$ is also pure.
\end{lem}
\begin{proof}
If $\varphi$ already has finite order in $\Aut(A_{\G})$, we can simply take $P=A_{\G}$ and there is nothing to prove. Suppose instead that, for some $n\geq 1$ and $g\in A_{\G}$, we have $\varphi^n(x)=gxg^{-1}$ for all $x\in A_{\G}$. We assume that this is a nontrivial automorphism, that is, that $g$ does not lie in the centre of $A_{\G}$. 

Since $\Fix(\varphi)\leq\Fix(\varphi^n)$, the centraliser $Z_{A_{\G}}(\Fix(\varphi))$ contains the element $g$. In particular, $Z_{A_{\G}}(\Fix(\varphi))$ is not contained in the centre of $A_{\G}$, from which it follows that the double centraliser $H:=Z_{A_{\G}}Z_{A_{\G}}(\Fix(\varphi))$ is a \emph{proper} subgroup of $A_{\G}$. Note that $\Fix(\varphi)\leq H$ and $\varphi(H)=H$. 

Being a centraliser, $H$ is convex-cocompact in $A_{\G}$ (i.e.\ it acts cocompactly on a convex subcomplex of the universal cover of the Salvetti complex) and it admits a splitting $H=P\x B$, where $P$ is a parabolic subgroup of $A_{\G}$ and $B$ is a convex-cocompact abelian subgroup of $A_{\G}$ (see e.g.\ Remark~3.7(5) and Corollary~3.25 in \cite{Fio10e}, or more historically, Servatius' Centralizer Theorem from \cite[Section~III]{Servatius}). 

Since $\varphi$ is untwisted, it preserves the coarse median structure on $A_{\G}$ (see \cite[Proposition~A]{Fio10a}), hence it preserves the induced coarse median structure on the invariant convex-cocompact subgroup $H=P\x B$. This implies that $\varphi(P)=P$ and $\varphi(B)=B$ and hence:
\[ \Fix(\varphi)=\Fix(\varphi|_H)=\Fix(\varphi|_P)\x\Fix(\varphi|_B).\]

Now, the group $\Fix(\varphi|_B)$ is free abelian. The automorphism $\varphi|_P$ preserves the coarse median structure on $P$, so it is still an untwisted automorphism of $P$ (this time by the other arrow in \cite[Proposition~A]{Fio10a}). The restriction $\varphi|_P$ still has finite order in $\Out(P)$; indeed, since $\varphi(H)=H$, it follows that $g$ normalises $P$, so we have $g\in P\cdot Z_{A_{\G}}(P)$ (by the structure of normalisers of parabolic subgroups of RAAGs) and hence $\varphi^n|_P$ is inner. Finally, if $\varphi$ was pure, so is $\varphi|_P$ by \cite[Lemma~3.29]{Fio10a}.

Summing up, $P<A_{\G}$ is a proper parabolic subgroup, the automorphism $\varphi|_P$ is untwisted, it has finite order projection to $\Out(P)$, and we have $\Fix(\varphi)=\Fix(\varphi|_P)\x\Z^k$ for some $k\geq 0$. Now, if $\varphi|_P$ actually has finite order in $\Aut(P)$ then the lemma is proved. Otherwise we replace $A_{\G}$ with $P$ and repeat the above procedure, which can only be necessary for a finite number of times, since each time the complexity of the RAAG strictly decreases.
\end{proof}

For the next lemma, we say that a $2$--sided hyperplane $H$ of a cube complex $X$ is \emph{inverted} by some $g\in\Aut(X)$ if, for any choice of orientation on $H$, we have $g(H)=H$ and $g(\ora{H})=\ola{H}$. Also recall that any intersection of hyperplanes of $X$ has a structure of a cube complex, as we discussed in Subsection~\ref{subsec:cube_complexes}.

\begin{lem}\label{lem:fix_structure}
Consider a non-positively curved cube complex $X$, a finite subgroup $\mf{F}\leq\Aut(X)$ and a connected component $C\sq\Fix(\mf{F})$.
\begin{enumerate}
\setlength\itemsep{.25em}
\item If a hyperplane $H\sq X$ is inverted by an element of $\mf{F}$ and $H\cap C\neq\emptyset$, then $C\sq H$. 
The intersection $\mc{I}\sq X$ of all such hyperplanes is $\mf{F}$--invariant. 
\item The component $C$ has a structure of a non-positively curved cube complex defined as follows:
	\begin{itemize}
	\item its $0$--skeleton is the intersection between $C$ and the $0$--skeleton of $\mc{I}$;
	\item each edge is the diagonal of an $\mf{F}$--invariant cube of $\mc{I}$ whose hyperplanes form a single (and entire) $\mf{F}$--orbit;
	\item higher-dimensional cubes are products of diagonals as in the previous point.
	\end{itemize}
\item If $X$ is special, then $C$ is special.
\end{enumerate}
\end{lem}
\begin{proof}
If an automorphism of $X$ inverts a hyperplane $H$, each of its fixed points is either in $H$ or outside the interior of the carrier of $H$. Part~(1) immediately follows from this observation.

Let us discuss part~(2). If $x\in\Fix(\mf{F})$ is a point and $c\sq X$ is the (only) cube whose interior contains $x$, then $c$ is $\mf{F}$--invariant. Thus, $\Fix(\mf{F})$ has a cell structure given by its intersections with $\mf{F}$--invariant cubes of $X$.

Now, consider one such cube $c$. Note that $c$ has an $\mf{F}$--invariant splitting $c=c_0\x c_1\x\dots\x c_k$, where each hyperplane of the cube $c_0$ is inverted by some element of $\mf{F}$, while the hyperplanes crossing each of the other $c_i$ form an $\mf{F}$--orbit $O_i$ (and are not inverted by any element of $\mf{F}$). It follows that $\Fix(\mf{F})\cap c=\{m\}\x d_1\x\dots\x d_k$, where $m$ is the centre of $c_0$ and each $d_i$ is a diagonal of $c_i$ (which simply means that $d_i=c_i$ if $\dim c_i=1$). This gives the cubical structure on $\Fix(\mf{F})$ described in the statement of part~(2). It is straightforward to check that this structure is non-positively curved, which completes the proof of part~(2). 

We now prove part~(3). Every hyperplane of $\Fix(\mf{F})$ is a connected component of the intersection between $\Fix(\mf{F})$ and a hyperplane of $X$. If $X$ is special, this immediately implies that hyperplanes of $\Fix(\mf{F})$ are embedded, $2$--sided and do not directly self-osculate. 

We are left to rule out inter-osculations. Consider a vertex $x$ of $\Fix(\mf{F})$ and two edges $d,d'$ stemming from $x$, which are diagonals of $\mf{F}$--invariant cubes $c,c'\sq\mc{I}$. Let $\{H_i\}$ and $\{H_j'\}$ be the hyperplanes of $X$ crossing $c$ and $c'$, respectively. If the hyperplanes of $\Fix(\mf{F})$ dual to $d$ and $d'$ cross each other, then this must occur in a cube of $X$ having all $H_i$ and all $H_j'$ as its hyperplanes. In particular, these hyperplanes pairwise cross each other and, by specialness of $X$, the cubes $c,c'$ span a cube $c\x c'\sq\mc{I}$ based at $x$. Since $c,c'$ are $\mf{F}$--invariant, so is $c\x c'$. In conclusion, $d$ and $d'$ span a square of $\Fix(\mf{F})$ at $x$. This rules out any inter-osculations in $\Fix(\mf{F})$.
\end{proof}

We are now ready to prove Proposition~\ref{prop:fix->divisible}. Inevitably, we will have to assume that the reader is familiar with \cite[Section~3]{CSV}.

\begin{proof}[Proof of Proposition~\ref{prop:fix->divisible}]
Products of divisible cube complexes are divisible. Thus, by Lemma~\ref{lem:finite-order_reduction}, it suffices to prove the proposition in its second case, i.e.\ for a finite subgroup $F<U^0(A_{\G})$.

By \cite[Theorem~8.1]{BCV2}, the projection to $\Out(A_{\G})$ of the subgroup $F$ can be realised as a finite group of cubical automorphisms of a Salvetti blowup of $A_{\G}$. In the notation and terminology of \cite{CSV}, let $\S_{\G}^{\mathbf{\Pi}}$ be this blowup, where $\mathbf{\Pi}=\{\mathbf{P}_1,\dots,\mathbf{P}_n\}$ is a compatible set of $\G$--Whitehead partitions. To avoid confusion, denote by $\mf{F}\leq\Aut(\S_{\G}^{\mathbf{\Pi}})$ the subgroup realising $F$.

By Lemma~\ref{lem:pi1_vs_fix}, we have $\Fix(F)\cong\pi_1(C)$ for a connected component $C\sq\Fix(\mf{F})$. Salvetti blowups are easily seen to be special, so by Lemma~\ref{lem:fix_structure} the subset $\Fix(\mf{F})$ has itself a structure of a compact special cube complex. We are left to show that the component $C$ is divisible.

The blowup $\S_{\G}^{\mathbf{\Pi}}$ has a hyperplane $H_v$ for each $v\in\G^{(0)}$, and a hyperplane $H_{\mathbf{P}_i}$ for each $\mathbf{P}_i\in\mathbf{\Pi}$. Each of these hyperplanes comes with natural orientations $H_v^+$ and $H_{\mathbf{P}_i}^+$. The crossing graph $\wh\G$ of the hyperplanes of $\S_{\G}^{\mathbf{\Pi}}$ thus contains a full subgraph isomorphic to $\G$ and additional vertices corresponding to $\mathbf{P}_1,\dots,\mathbf{P}_n$.

Every hyperplane of $C$ is a connected component of an intersection $H\cap C$, where $H$ is a hyperplane of $\S_{\G}^{\mathbf{\Pi}}$. Given two hyperplanes $H,K\in\mscr{W}(\S_{\G}^{\mathbf{\Pi}})$, we have $H\cap C=K\cap C\neq\emptyset$ exactly when $H$ and $K$ cross $C$ and lie in the same $\mf{F}$--orbit (recall Lemma~\ref{lem:fix_structure}), in which case they are necessarily transverse as hyperplanes of $\S_{\G}^{\mathbf{\Pi}}$. Otherwise, $H\cap C$ and $K\cap C$ are disjoint.

Let $\L$ be the quotient of $\wh\G$ by its induced $\mf{F}$--action and let $\Lambda_C\sq\Lambda$ be the subgraph corresponding to $\mf{F}$--orbits of hyperplanes that cross $C$. We have a natural special colouring $\kappa\colon\mscr{W}(C)\ra\Lambda_C$ pairing each hyperplane of $C$ with the $\mf{F}$--orbit of hyperplanes of $\S_{\G}^{\mathbf{\Pi}}$ of which it is the intersection with $C$. We define the map $\kappa\colon\mscr{O}(C)\ra\Lambda_C^{\pm}$ so that the orientations $H_v^+$ and $H_{\mathbf{P}_i}^+$ determine oriented hyperplanes of $C$ with positive colour. 

Each $\mathbf{P}_i\in\mathbf{\Pi}$ induces a partition $\S_{\G}^{\mathbf{\Pi}}=\mc{H}_i^-\sqcup\mc{H}_i^0\sqcup\mc{H}_i^+$, where $\mc{H}_i^0$ is the union of the hyperplane $H_{\mathbf{P}_i}$ and all hyperplanes $H_v$ with $v\in\lk(\mathbf{P}_i)$. We can restrict this partition to an analogous partition $\mc{H}_i$ of $C$ (provided that the restriction is nontrivial).

From the construction of blowups in \cite[Section~3.3]{CSV}, it is completely straightforward to check that the special colouring $(\kappa,\Lambda_C)$ and the partitions $\mc{H}_i$ satisfy all conditions in Definition~\ref{defn:dividing_pattern} and thus form a dividing pattern for $C$. We only discuss the second half of condition~(4) in the next paragraph, since it is the most delicate.

Suppose that a vertex $x\in C$ lies outside the generalised half-carrier $C(w^+)\sq C$ for some $w\in\Lambda_C$. Then $x$ must lie outside all carriers $C(H_u^+)\sq \S_{\G}^{\mathbf{\Pi}}$ where $u\in\wh\G$ is a colour projecting to $w\in\Lambda_C$. Indeed, for any such hyperplane $H_u$, the orbit $\mf{F}\cdot H_u$ is pairwise-transverse in $S_{\G}^{\mathbf{\Pi}}$; thus, if $x$ lied in some $C(H_u^+)$, then it would be a vertex of a cube of $\S_{\G}^{\mathbf{\Pi}}$ with hyperplane set $\mf{F}\cdot H_u$, and the diagonal of this cube would be contained in $\Fix(\mf{F})$, contradicting the fact that $x\not\in C(w^+)\sq C$. Now, since $x\not\in C(H_u^+)$ in $\S_{\G}^{\mathbf{\Pi}}$ for some (and in fact all) $u\in\wh\G$ projecting to $w$, there exists a partition $\mc{H}_i$ such that $x\in\mc{H}_i^-$ and $C(H_u^+)\sq\mc{H}_i^+$ (or vice versa), straight from the construction of $\S_{\G}^{\mathbf{\Pi}}$ in \cite{CSV}. Hence $x\in\mc{H}_i^-$ and $C(w^+)\sq\mc{H}_i^+$ for the partition restricted to $C$, as required.

In conclusion, the component $C\sq\Fix(\mf{F})$ is divisible, completing the proof of the proposition.
\end{proof}

\section{Automorphisms of finite-index subgroups of RAAGs}

This section is devoted to the proof of Proposition~\ref{propintro:virtual_aut}. The argument is based on a small variation of the Haglund--Wise construction of canonical completions \cite{HW08} and occupies Subsection~\ref{subsec:virt_RAAG}. 

Subsection~\ref{subsec:virt_RACG} then briefly describes a further adaptation realising special groups as fixed subgroups of automorphisms of finite-index subgroups of right-angled Coxeter groups. That this is possible is also an immediate consequence of the fact that RAAGs are finite-index subgroups of RACGs \cite{Davis-Janusz}, the main advantage of our second argument being that it requires smaller RACGs and smaller-index subgroups.

The main results of the section are Corollary~\ref{cor:virtual_RAAG} and Corollary~\ref{cor:virtual_RACG}.

\subsection{Multiple octahedralisations}

In this subsection, we fix notation for a construction that will be needed in the following discussion. Let $\G$ be a finite simplicial graph and let $N\geq 1$ be an integer.

\begin{defn}\label{defn:multiple_octahedralisation}
The \emph{$N$--th octahedralisation} of $\G$ is the graph $\G[N]$ defined as follows. For every vertex $v\in\G$, there are $2N$ vertices of $\G[N]$, denoted $v_1^-,\dots,v_N^-$ and $v_1^+,\dots,v_N^+$. For every edge $[v,w]\sq\G$, every $\eps,\eta\in\{\pm\}$ and every $1\leq i,j\leq N$, the graph $\G[N]$ has an edge $[v_i^{\eps},w_j^{\eta}]$.
\end{defn}

In particular, for every vertex $v\in\G$, the $2N$ vertices $v_i^{\pm}$ form an anti-clique. Note that these $2N$ vertices are completely interchangeable --- there are automorphisms of $\G[N]$ realising any permutation of the $v_i^{\pm}$ while fixing the rest of $\G[N]^{(0)}$ pointwise --- and the superscript signs do not reflect any particular properties. Our notation was chosen to simplify the coming discussion.

The graph $\G[1]$ is commonly known as the \emph{octahedralisation} of $\G$. It is the ($1$--skeleton of the) link of the vertex of the Salvetti complex for the right-angled Artin group $A_{\G}$.

It is also convenient to introduce the following notation. (Again, the reason for our weird-looking choices will become clear in the next subsection.)

\begin{defn}
The graph $\G[N/2]$ is obtained from $\G[N]$ by removing all vertices with a $-$ superscript. The $+$ superscripts are then dropped from the remaining vertices (so that the vertices previously denoted $v_i^+$ are now simply denoted $v_i$).
\end{defn}

In particular, $\G[1/2]$, $\G[2N/2]$, $\G[N/2][1]$ are naturally identified with $\G$, $\G[N]$, $\G[N]$, respectively.

\subsection{Finite-index subgroups of RAAGs}\label{subsec:virt_RAAG}

Let $Q$ be a compact special cube complex. Simply denote by $\G$ the crossing graph of the hyperplanes of $Q$ (denoted by $\G_Q$ in the rest of the paper). 

To simplify the following discussion, we will repeatedly speak of ``the link'' of vertices of $Q$ (and other cube complexes), even though we will actually be referring to the $1$--skeletons of said links.

Choose an orientation for each hyperplane of $Q$ and label edge germs accordingly. That is, if an (oriented) edge $e\sq Q$ crosses a hyperplane $H\sq Q$ in the direction prescribed by the chosen orientation of $H$, we label by $H^+$ the germ of $e$ at its initial vertex, and by $H^-$ the germ of $e$ at its terminal vertex. With these labels, the link of every vertex of $Q$ admits a fixed label-preserving embedding into the octahedralisation $\G[1]$ (we do not distinguish between $H^{\pm}$ and $H_1^{\pm}$, which would be the actual labels in $\G[1]$ according to Definition~\ref{defn:multiple_octahedralisation}).

We will now embed $Q$ as a locally convex subcomplex of a compact special cube complex $\wh Q$ such that $\wh Q^{(0)}=Q^{(0)}$ and all links of vertices of $\wh Q$ are naturally identified with the graph $\G[3]$. We will then construct a cubical automorphism of $\wh Q$ having $Q$ as its fixed set.

The construction of $\wh Q$ is a minor variation of the \emph{canonical completion} of Haglund and Wise \cite[Section~6]{HW08}. We begin by adding edges to $Q$, labelling their germs by vertices of $\G[3]$. Consider a vertex $x\in Q^{(0)}$ and a hyperplane $H\in\G^{(0)}$. We distinguish three cases, depending on the local structure of $Q$ at $x$.
\begin{enumerate}
\setlength\itemsep{.25em}
\item \emph{There are no edge germs labelled $H^-$ or $H^+$ at $x$.} In this case, we add three loops at $x$: one with germs labelled $H_1^{\pm}$, one with germs labelled $H_2^{\pm}$ and one with germs labelled $H_3^{\pm}$.
\item \emph{There are edge germs labelled $H^-$ and $H^+$ at $x$.} In this case, let $x'$ and $x''$ be the other vertices of the two edges whose germs at $x$ are labelled $H^-$ and $H^+$, respectively. We add two edges between $x$ and $x'$, one with germ at $x$ labelled $H_2^-$ and germ at $x'$ labelled $H_2^+$, and another with germ at $x$ labelled $H_3^-$ and germ at $x'$ labelled $H_3^+$. We similarly add two more edges between $x$ and $x''$, with germs at $x$ labelled $H_2^+,H_3^+$ and germs at $x''$ labelled $H_2^-,H_3^-$. (It is possible that $x'=x''=x$, in which case we just add two loops at $x$, rather than four edges.) We think of the existing labels $H^{\pm}$ as $H_1^{\pm}$.
\item \emph{There is an edge germ labelled $H^-$ at $x$, but none labelled $H^+$ (or vice versa).} In this case, we first define $x'$ as in case~(2) and similarly add two edges between $x$ and $x'$, one with germ at $x$ labelled $H_2^-$ and germ at $x'$ labelled $H_2^+$, and another with germ at $x$ labelled $H_3^-$ and germ at $x'$ labelled $H_3^+$.

Then we define a third vertex $y$ as follows. Let $\g\sq Q$ be the maximal edge path starting at $x$ such that all its edges cross the hyperplane $H$ (and such that its edges are distinct). We define $y$ as the terminal vertex of $\g$. Note that $y$ has the ``opposite problem'' compared to $x$: it has an edge germ labelled $H^+$, but none labelled $H^-$. Thus, we add three edges between $x$ and $y$: one with germ $H_i^+$ at $x$ and germ $H_i^-$ at $y$ for each value of $1\leq i\leq 3$.
\end{enumerate}
We repeat the above construction for all vertices $x\in Q^{(0)}$ and hyperplanes $H\in\G^{(0)}$. 

We orient all edges from the germ with positive superscript to the germ with negative superscript. For simplicity, we will speak of \emph{$H_i$--edges}, for some $H\in\G^{(0)}$ and $1\leq i\leq 3$, referring to any edge whose germs are labelled $H_i^{\pm}$.

The $1$--skeleton of $\wh Q$ will consist of the $1$--skeleton of $Q$ along with the edges that we have just added. Before we glue in higher-dimensional cubes, we make following observation.

\begin{rmk}\label{rmk:same_transverse_labels}
We draw the reader's attention to a couple of properties of the above construction.
\begin{enumerate}
\setlength\itemsep{.25em}
\item Suppose that two vertices $x,y\in \wh Q^{(0)}=Q^{(0)}$ are joined by an $H_i$--edge of $\wh Q$, for some $H\in\G^{(0)}$ and $1\leq i\leq 3$. If a hyperplane $K\in\G^{(0)}$ is transverse to $H$, then there is an edge germ labelled $K^+$ at $x$ (in the original $Q$) if and only if there is an edge germ labelled $K^+$ at $y$ (in the original $Q$). Similarly, $K^-$ appears at $x$ if and only if $K^-$ appears at $y$.

This is clear in cases~(1) and~(2) above. In case~(3), this is because, if $x$ lies in the carrier of the hyperplane $K$, then the path $\g$ is entirely contained in the same carrier, and so is the vertex $y$.

\item Consider a vertex $x\in\wh Q$ and some $H\in\G^{(0)}$. There are three edges in $\wh Q$ whose germs at $x$ are labelled $H_1^+,H_2^+,H_3^+$. The other germs of these three edges (the ones labelled $H_i^-$) are all based at the same vertex, which we will simply denote by $H^+\cdot x$. The exact same observation holds swapping the signs $\pm$, leading to the definition of a vertex $H^-\cdot x$. 

We always have $H^-\cdot(H^+\cdot x)=H^+\cdot(H^-\cdot x)=x$. If $H,K\in\G^{(0)}$ are transverse, we also have $K^+\cdot(H^+\cdot x)=H^+\cdot(K^+\cdot x)$ and $K^-\cdot(H^+\cdot x)=H^+\cdot(K^-\cdot x)$ (and so on). Note that it is possible for $H^{\pm}\cdot x$ to coincide with $x$.  
\end{enumerate}
\end{rmk}

Now, consider the $1$--skeleton of $\wh Q$ described above. We attach a square filling any $4$--cycle in which one pair of opposite edges are $H_i$--edges oriented in the same direction, for some $H\in\G^{(0)}$ and $1\leq i\leq 3$, and the other pair of opposite edges are $K_j$--edges oriented in the same direction, for some $K\in\G^{(0)}$ transverse to $H$ and some $1\leq j\leq 3$. Any such $4$--cycle has vertex set of the form $\{x,H^+\cdot x,K^+\cdot(H^+\cdot x)=H^+\cdot(K^+\cdot x),K^+\cdot x\}$ or $\{x,H^+\cdot x,K^-\cdot(H^+\cdot x)=H^+\cdot(K^-\cdot x),K^-\cdot x\}$.

Finally, we add the higher-dimensional cubes required to ensure that the end result $\wh Q$ is non-positively curved.

\begin{prop}\label{prop:completion_properties}
\begin{enumerate}
\item[]
\item The cube complex $\wh Q$ is compact, special and all its vertex links have a label-preserving isomorphism with $\G[3]$. The subcomplex $Q\sq\wh Q$ is locally convex.
\item There is a cubical isomorphism $\Phi\colon\wh Q\ra\wh Q$ such that $\Phi^6=\id$ and $Q=\Fix(\Phi)$.
\item The fundamental group of $\wh Q$ embeds in the RAAG $A_{\G[3/2]}$ as a subgroup of index $\#Q^{(0)}$.
\end{enumerate}
\end{prop}
\begin{proof}
Part~(1) is immediate from the above discussion (see especially Remark~\ref{rmk:same_transverse_labels}).

We now prove part~(2). The automorphism $\Phi$ fixes pointwise $\wh Q^{(0)}=Q^{(0)}$, as well as each edge of $Q\sq\wh Q$. On the rest of the $1$--skeleton of $\wh Q$, we define $\Phi$ as follows. Let $e_1,e_2,e_3$ be the three $H_i$--edges with a given pair of endpoints and orientation. If $e_1\sq Q$, so that it is fixed by $\Phi$, then $\Phi$ swaps $e_2$ and $e_3$. Otherwise, $\Phi$ cyclically permutes the edges by $e_1\mapsto e_2\mapsto e_3\mapsto e_1$.

Note that $\Phi$ preserves the orientation of every edge of $\wh Q$, and it takes each $H_i$--edge to an $H_j$--edge. Since $\wh Q$ is special, it suffices to show that $\Phi$ takes square boundaries to square boundaries in order to prove that $\Phi$ extends to an automorphism of $\wh Q$. That $\Phi$ indeed has this property quickly follows from part~(1) of Remark~\ref{rmk:same_transverse_labels} (one only needs to observe that it is not possible for opposite edges in a square to have different-cardinality orbits under $\Phi$). It is also clear that $Q=\Fix(\Phi)$. 

Finally, we prove part~(3). The link of the Salvetti complex $\mbb{S}$ for $A_{\G[3/2]}$ is naturally identified with $\G[3/2][1]=\G[3]$. Thus, there is a natural locally isometric immersion $\wh Q\looparrowright\mbb{S}$ that preserves labels at the level of links. In fact, this map is also locally surjective, so it must be a covering map. It follows that $\pi_1(\wh Q)$ is isomorphic to a subgroup of $A_{\G[3/2]}$ of index $\#\wh Q^{(0)}=\#Q^{(0)}$, as required.
\end{proof}

Along with Lemma~\ref{lem:pi1_vs_fix}(2), this proves Proposition~\ref{propintro:virtual_aut}.

\begin{cor}\label{cor:virtual_RAAG}
Let $Q$ be a compact special cube complex and let $\G$ be the crossing graph of its hyperplanes. Then there exist a subgroup $H\leq A_{\G[3/2]}$ of index exactly $\#Q^{(0)}$ and a coarse-median preserving automorphism $\varphi\in\Aut(H)$ such that $\varphi^6=\id_H$ and $\Fix(\varphi)\cong\pi_1(Q)$.
\end{cor}

\subsection{Finite-index subgroups of RACGs}\label{subsec:virt_RACG}

Since every RAAG embeds as a finite-index subgroup of a RACG \cite{Davis-Janusz}, Corollary~\ref{cor:virtual_RAAG} also implies that every special group can be realised as the fixed subgroup of an automorphism of a finite-index subgroup of a RACG. In this subsection, we adapt the argument used in the previous subsection to reprove this fact, under slightly stronger assumptions on $Q$ but also obtaining better control on the size of the RACG and on the index of its subgroup.

Let again $Q$ be a compact special cube complex. We now also assume that the hyperplanes of $Q$ have no indirect self-osculations, which can always be ensured by passing to a finite cover (see \cite[Proposition~3.10]{HW08}). 

Let $\pi\colon Q_2\ra Q$ be any double cover of $Q$. That one such cover exists follows e.g.\ from the fact that RAAGs (and therefore special groups) are residually $2$--finite. See \cite[Theorem~6.1]{Toinet}, who deduces this from \cite{Duchamp-Krob,Gruenberg}.

Let $\G$ be the crossing graph of the hyperplanes of $Q$. We label every edge $e\sq Q$ by the hyperplane $H\in\G^{(0)}$ that it crosses; we also label by $H$ the two germs of $e$ at its endpoints (disregarding orientations). Via the covering map $\pi$, we pull this labelling back to a labelling by $\G$ of the edges and edge germs of $Q_2$. In particular, the link of any vertex of $Q_2$ admits a natural label-preserving embedding into $\G$ (as in the previous subsection, we will again speak of ``links'' even though we are actually referring to their $1$--skeleton).

Similarly to the previous subsection, we will embed $Q_2$ as a locally convex subcomplex of a compact special cube complex $\wh Q_2$ such that $\wh Q_2^{(0)}=Q_2^{(0)}$ and all links of vertices of $\wh Q_2$ are naturally identified with the graph $\G[3/2]$. We will then construct a cubical automorphism of $\wh Q_2$ having $Q_2$ as its fixed set.

We begin by adding edges to $Q_2$, labelling them and their germs by vertices of $\G[3/2]$. Consider a hyperplane $H\in\G^{(0)}$, a vertex $x\in Q$ and the two vertices $x_1,x_2\in\pi^{-1}(x)\sq Q_2$. Note that there are label-preserving isomorphisms between the links of $x,x_1,x_2$. We distinguish two cases.
\begin{enumerate}
\setlength\itemsep{.25em}
\item \emph{There is no edge labelled $H$ stemming from $x$ in $Q$.} In this case, we add to $Q_2$ three edges between $x_1$ and $x_2$, labelling them $H_1,H_2,H_3$, respectively.
\item \emph{There is an edge $e\sq Q$ labelled $H$ stemming from $x$.} In this case, there are two lifts $e_1,e_2\sq Q_2$ of the edge $e$, with endpoints $x_1,y_1$ and $x_2,y_2$, respectively. We then add four edges to $Q_2$: two between $x_1$ and $y_1$ with labels $H_2$ and $H_3$; and two between $x_2$ and $y_2$, also with labels $H_2$ and $H_3$. (We identify any existing labels $H$ with $H_1$.)
\end{enumerate}

This completes the $1$--skeleton of $\wh Q_2$. Now, we attach a square to each $4$--cycle in which the two pairs of opposite edges are labelled $H_i$ and $K_j$, respectively, for any pair of transverse hyperplanes $H,K\in\G^{(0)}$ and indices $1\leq i,j\leq 3$. Then we add the higher-dimensional cubes required to ensure that the end result $\wh Q_2$ is non-positively curved.

The proof of the following result is entirely analogous to that of Proposition~\ref{prop:completion_properties} (and, in fact, a little simpler).

\begin{prop}
\begin{enumerate}
\item[]
\item The cube complex $\wh Q_2$ is compact, special and all its vertex links have a label-preserving isomorphism with $\G[3/2]$. The subcomplex $Q_2\sq\wh Q_2$ is locally convex.
\item There is a cubical isomorphism $\Phi\colon\wh Q_2\ra\wh Q_2$ such that $\Phi^6=\id$ and $Q_2=\Fix(\Phi)$.
\item The fundamental group of $\wh Q_2$ embeds into the RACG $W_{\G[3/2]}$ as a subgroup of index $\# Q_2^{(0)}$. 
\end{enumerate}
\end{prop}

Again, this and Lemma~\ref{lem:pi1_vs_fix}(2) prove the following analogue of Corollary~\ref{cor:virtual_RAAG}.

\begin{cor}\label{cor:virtual_RACG}
Let $Q$ be a compact special cube complex with no (directly or indirectly) self-osculating hyperplanes. Let $\G$ be the crossing graph of the hyperplanes of $Q$. Then there exist an index--$2$ subgroup $G_0<\pi_1(Q)$, a subgroup $H\leq W_{\G[3/2]}$ of index $2\cdot\#Q^{(0)}$ and a coarse-median preserving automorphism $\varphi\in\Aut(H)$ such that $\varphi^6=\id$ and $\Fix(\varphi)\cong G_0$.
\end{cor}

\bibliography{./mybib}
\bibliographystyle{alpha}

\end{document}